\documentclass[12pt]{amsart}
\usepackage[left=3.5cm, right=3cm, top=4cm, bottom=4cm]{geometry} 
\usepackage{
 amsmath, 
 amsxtra, 
 amsthm, 
 amssymb, 
 etex, 
 mathrsfs, 
 mathtools, 
 tikz-cd, 
 bbm,
 xr,
 comment,
 enumitem}
\usepackage[all]{xy}
\usepackage{hyperref}
\usepackage{url}
\usepackage{tipa}
\usepackage{dsfont}

\newtheorem{theorem}{Theorem}[section]
\newtheorem{lemma}[theorem]{Lemma}
\newtheorem{conjecture}[theorem]{Conjecture}
\newtheorem{proposition}[theorem]{Proposition}

\newtheorem{defn}[theorem]{Definition}

\newtheorem{conj}[theorem]{Conjecture}
\newtheorem{hyp}[theorem]{Hyp}
\numberwithin{equation}{section}
\newtheorem{lthm}{Theorem} % theorems with letters (for intro)
\newtheorem{propl}[lthm]{Proposition}

\theoremstyle{remark}
\newtheorem{remark}[theorem]{Remark}
\newtheorem{conv}[theorem]{Convention}
\usepackage[utf8]{inputenc}

\newcommand\EatDot[1]{}

\newcommand{\cyc}{{\mathrm{cyc}}}

\newcommand{\cL}{{\mathcal{L}}}
\newcommand{\bH}{\mathbf{H}}

\newcommand{\sL}{{\mathscr{L}}}
\newcommand{\fL}{{\mathfrak{L}}}

\newcommand{\cH}{{\mathcal{H}}}

\newcommand{\Mlog}{M_{\log}}

\newcommand{\Span}{\mathrm{Span}}

\newcommand{\Fil}{\mathrm{Fil}}
\newcommand{\Image}{\mathrm{Image}}

\newcommand{\ch}{{\mathrm{char}}}

\newcommand{\Gal}{{\mathrm{Gal}}}
\newcommand{\Ker}{{\mathrm{Ker}}}

\newcommand{\Sel}{{\mathrm{Sel}}}
\newcommand{\GL}{{\mathrm{GL}}}

\newcommand{\Tw}{\mathrm{Tw}}

\newcommand{\Hom}{\mathrm{Hom}}
\newcommand{\Lie}{\mathrm{Lie}}
\newcommand{\loc}{\mathrm{loc}}

\newcommand{\ord}{\mathrm{ord}}
\newcommand{\Q}{{\mathbb Q}}
\newcommand{\Z}{{\mathbb Z}}
\newcommand{\bz}{{\mathbf z}}

\newcommand{\A}{{\mathbb A}}
\newcommand{\R}{{\mathbb R}}

\newcommand{\C}{\mathbb{C}}
\newcommand{\vp}{\varphi}
\newcommand{\cc}{\chi_\cyc}
\newcommand{\Brig}{\mathbb{B}_{\mathrm{rig}}^+}
\newcommand{\Dcris}{\mathbb{D}_{\mathrm{cris}}}

\newcommand{\col}{\mathrm{Col}}
\newcommand{\QQ}{\mathbb{Q}}
\newcommand{\ZZ}{\mathbb{Z}}
\newcommand{\Qp}{\mathbb{Q}_p}
\newcommand{\Zp}{\mathbb{Z}_p}
\newcommand{\Cp}{\mathbb{C}_p}
\newcommand{\NN}{\mathbb{N}}
\newcommand{\lb}{[[}
\newcommand{\rb}{]]}
\newcommand{\cN}{\mathcal{N}}
\newcommand{\tell}{\widetilde{\ell}}
\newcommand{\Crit}{\mathrm{Crit}}
\newcommand{\rank}{\mathrm{rank}}
\newcommand{\HIw}{H^1_{\mathrm{Iw}}}
\definecolor{Green}{rgb}{0.0, 0.5, 0.0}
\newcommand{\bZ}{\mathbf{Z}}
\newcommand{\cO}{\mathcal{O}}
\newcommand{\tcol}{\widetilde{\col}}

\title[Non-ordinary Iwasawa theory of automorphic representations]{Iwasawa theory of automorphic representations of $\mathrm{GL}_{2n}$ at non-ordinary primes}

\author[A.~Lei]{Antonio Lei}
\address[Lei]{Department of Mathematics and Statistics\\University of Ottawa\\
150 Louis-Pasteur Pvt\\
Ottawa, ON\\
Canada K1N 6N5}
\email{antonio.lei@uottawa.ca}

\author[J. Ray]{Jishnu Ray}
\address[Ray]{TCG Centres for Research and Education
in Science and Technology \\1st Floor, Tower 1, Bengal Eco Intelligent Park (Techna Building), Block EM, Plot No 3, Sector V, Salt Lake, Kolkata 700091}
\email{jishnu.ray@tcgcrest.org; jishnuray1992@gmail.com}

\keywords{automorphic representations, non-ordinary primes, $p$-adic $L$-functions, Iwasawa main conjecture}

\subjclass[2020]{Primary: 11R23; Secondary: 11F70, 11F80, 11F67}

\begin{document}

\begin{abstract}
    Let $\Pi$ be a cuspidal automorphic representation of $\mathrm{GL}_{2n}(\mathbb{A_Q})$ and let $p$ be an odd prime at which $\Pi$ is unramified. In a {recent} work, Barrera, Dimitrov and Williams {constructed possibly unbounded}  $p$-adic $L$-functions interpolating complex $L$-values of $\Pi$ in the non-ordinary case. {Under certain assumptions}, we  construct two \textit{bounded} $p$-adic $L$-functions for $\Pi$,  thus extending an earlier work of  Rockwood by relaxing the Pollack condition. Using Langlands local-global compatibility,   we define signed Selmer groups over the $p$-adic cyclotomic extension of $\QQ$ attached to the $p$-adic Galois representation of $\Pi$ and formulate Iwasawa main conjectures in the spirit of Kobayashi's plus and minus main conjectures for $p$-supersingular elliptic curves.

\end{abstract}

\maketitle

\section{Introduction}
\subsection*{A brief history on non-ordinary Iwasawa theory of elliptic curves and elliptic modular forms} Let $p$ be a fixed odd prime and $f=\sum_{m\ge1} a_mq^m$ a cuspidal elliptic modular form. When $f$ has  good ordinary reduction at $p$, Mazur and Swinnerton-Dyer, as well as Manin, constructed a $p$-adic $L$-function attached to $f$ which is a bounded measure on $\Gamma:=\Gal(\QQ(\mu_{p^\infty})/\QQ)$, interpolating complex $L$-values of $f$ twisted by Dirichlet characters on $\Gamma$  (see \cite{MSD74,manin}). When $f$ has good non-ordinary reduction at $p$, the situation is quite different. While the aforementioned works on the ordinary case can be extended to the non-ordinary case, the resulting $p$-adic $L$-functions are not necessarily bounded  (see \cite{amicevelu75,visik76}). 

In  \cite{pollack03}, Pollack showed that when $a_p=0$, there is a very elegant way to decompose the $p$-adic $L$-functions attached to $f$, using the so-called plus and minus logarithms, into bounded measures by exploiting the symmetry between the two roots of the Hecke polynomial of $f$ at $p$. These bounded measures were  utilized by Kobayashi to formulate the so-called plus and minus Iwasawa main conjectures in \cite{kobayashi03} when $f$ corresponds to an elliptic curve defined over $\QQ$. More precisely, Kobayashi defined the so-called plus and minus Selmer groups over $\QQ(\mu_{p^\infty})$. Using Kato's Euler system constructed in \cite{Kato}, he showed that their Pontryagin duals are torsion over the Iwasawa algebra of $\Gamma$ and related the characteristic ideals to Pollack's bounded $p$-adic $L$-functions. Kobayashi's work has been generalized to {higher weight elliptic} modular forms when $a_p=0$ by the first named author of this article (see \cite{lei11compositio}). The works of Pollack and Kobayashi have been simultaneously generalized to the $a_p\ne 0$ case by Sprung \cite{sprung09} (for elliptic curves) and by Lei--Loeffler--Zerbes \cite{leiloefflerzerbes10,leiloefflerzerbes11} (for {higher weight elliptic} modular forms).
\subsection*{Automorphic   results in this paper}
Let $n \geq 1$ be an integer and $\Pi$  a  {regular algebraic, essentially self-dual, cuspidal} automorphic representation of $\GL_{2n}(\A_\Q)$ that is unramified at $p$. When $\Pi$ admits a Shalika model and is ordinary at $p$,  Dimitrov, Januszewski and Raghuram generalized earlier results of Ash and Ginzburg in \cite{AshGinzburg}  to construct a bounded $p$-adic $L$-function interpolating the complex $L$-values of $\Pi$ twisted by Dirichlet characters on $\Gamma$ (see \cite{DJR}). 
This construction {has been} further  generalized to the non-ordinary case in a {recent work of Barrera, Dimitrov and Williams \cite{BDW}}. As in the case of elliptic modular forms, when $\Pi$ is non-ordinary at $p$, the resulting $p$-adic $L$-functions  are  distributions {that are possibly unbounded}.

{Thanks to} results from local Langlands program, there exists a compatible family $\rho_{\Pi,\lambda, \iota}: G_\Q \rightarrow \GL_{2n}(\mathrm{E}(\Pi)_\lambda)$ of continuous $\lambda$-adic representations where $\lambda$ runs through the finite places of a number field $\mathrm{E}(\Pi)$ (see Section \ref{sec:RASEDC} below). From now on, we fix a prime $\lambda$ of $E(\Pi)$ lying above $p$ with ring of integers $\cO$ and write $F=E(\Pi)_\lambda$. By enlarging $F$ if necessary, we shall assume that $F$ contains all the Satake parameters of $\Pi$ at $p$. Let 
$V_\Pi$ be the dual representation {$\rho_{\Pi,\lambda,\iota}^*$}. We fix {a $G_{\QQ}$}-stable lattice  $T_\Pi$ inside $V_\Pi$ and write $T_\Pi^\dagger=\Hom_{\mathrm{cts}}(T_\Pi, F/\cO(1))$.
The main results of this paper can be summarized as follows.

\begin{propl}[Proposition \ref{thm:upper-tri}] \label{MT}
Suppose that the hypotheses in \text{Hyp \ref{hypSet1}} hold. The local representation {$V_\Pi|_{G_{\Qp}}$} is of the form 
\[
\begin{pmatrix}
\cc^{h_{1}}\theta_{1} &*&\cdots&*&*&*&*&\cdots&*&*\\
0&\cc^{h_{2}}\theta_{2} &\cdots&*&*&*&*&\cdots&*&*\\
\vdots&\ddots&\ddots&&\vdots&\vdots&\vdots&&\vdots&\vdots\\
0&\cdots&0&\cc^{h_{n-1}}\theta_{n-1}&*&*&*&\cdots&*&*\\
0&\cdots&\cdots&0 &\times&\times&*&\cdots&*&*\\
0&\cdots&\cdots&0 &\times&\times&*&\cdots&*&*\\
0&\cdots&\cdots&0 &0&0&\cc^{h_{n+2}}\theta_{n+2}&\cdots&*&*\\
\vdots&&& \vdots&\vdots&\vdots&\ddots&\ddots&\vdots&\vdots\\
0&\cdots&\cdots&0 &0&0&\cdots&0&\cc^{h_{2n-1}}\theta_{2n-1}&*\\
0&\cdots&\cdots&0 &0&0&\cdots&\cdots&0&\cc^{h_{2n}}\theta_{2n}
\end{pmatrix},
\]
where $\cc$ is the $p$-adic cyclotomic character, $h_i$  are the Hodge--Tate weights {of $V_\Pi$} and  $\theta_i$ are unramified characters on $G_{\Qp}$.
\end{propl}

{Proposition} ~\ref{MT} generalizes a result of Ghate--Kumar in \cite{GhateKumarPublished} where they showed that the local representation attached to $\Pi$ is upper-triangular in the ordinary case.
The conditions in Hyp \ref{hypSet1} assert certain relations between the Hodge--Tate weights and  Satake paremeters of $\Pi$ at $p$ and that the Newton and Hodge filtrations of the Dieudonné module are in general position. 
More specifically, if $\alpha_1,\ldots \alpha_{2n}$ are the Satake parameters ordered by $\ord_p(\alpha_1)\ge \ord_p(\alpha_2)\ge\cdots\ge \ord_p(\alpha_{2n})$, we are assuming that
\begin{align*}
   &\ord_p(\alpha_i)=h_i,\ i=n+2,\ldots, 2n; \\
   &\ord_p(\alpha_{n+1})>h_{n+1}\ \text{and}\ \alpha_n\ne\alpha_{n+1}.
\end{align*}
These conditions ensure that the local representation is close to being ordinary, admitting at most a 2-dimensional non-ordinary sub-quotient. It allows us to modify previous works on non-ordinary Iwasawa theory for elliptic modular forms to the current setting. In particular, we prove:

\begin{lthm}[Theorem~\ref{thm:sharpflatpadicL}]\label{MT2}
 Suppose that the hypotheses  \text{Hyp \ref{hypSet1}} and \text{Hyp \ref{hypSet2}} hold.  Let $\sL_p^{(\alpha)}$ and $\sL_p^{(\beta)}$ be certain twists of the  $p$-adic $L$-functions of Barrera, Dimitrov and Williams (see Definition \ref{twistBWD}). There exist bounded $p$-adic $L$-functions $\sL_p^\#,\sL_p^\flat \in \cO[[\Gamma]] \otimes F$  such that
\[
\begin{pmatrix}
   \sL_p^{(\alpha)}\\ \sL_p^{(\beta)}
   \end{pmatrix}=Q^{-1}\Mlog'\begin{pmatrix}
 \sL_p^\#\\  \sL_p^\flat
   \end{pmatrix},
\]
where $Q$ and $\Mlog'$ are certain $2\times 2$ matrices with coefficients in the the distribution algebra on $\Gamma$, defined in Definitions \ref{defn:twist-col} and \ref{def:ColBlock} respectively. Furthermore, at least one of the two $p$-adic $L$-functions $ \sL_p^\#$ and $ \sL_p^\flat$ is non-zero.
\end{lthm}
{The conditions in Hyp \ref{hypSet2} assert that the unramified characters appearing  in  the 1-dimensional  sub-quotients in Proposition~\ref{MT} are non-trivial and that the 2-dimensional sub-quotient in the middle is non-ordinary and satisfies  the  Fontaine--Laffaille condition. The former condition simplifies some of our calculations with local cohomology groups  and  can potentially be removed with some extra work, whereas the latter is crucial in order for us to apply previous results on  Wach modules theory in our  construction of the matrix $\Mlog'$.}

 Theorem~\ref{MT2} generalizes a prior result of Rockwood \cite{rob}. One of the main hypotheses in \cite{rob} is the Pollack condition, which says that $\alpha_n +  \alpha_{n+1}=0$, where $\alpha_j$ is the $j$-th Satake parameter of $\Pi$ at $p$, ordered according to their $p$-adic valuations. This Pollack condition is analogous to the condition $a_p=0$ for {elliptic} modular forms, which means that the two roots of its Hecke polynomial at $p$ add up to zero. Our main effort in this article is to construct bounded $p$-adic $L$-functions without assuming the Pollack condition.

 Note that in \cite{BDW}, the authors in fact constructed $p$-adic $L$-functions for automorphic representations over a totally real field. However, the Wach module theory that we rely on is only available for unramified extensions of $\Qp$. It will be interesting to study a generalization of Theorem~\ref{MT2} for automorphic representations over a totally real field where $p$ is unramified. Assuming the local representation of $\Pi$ satisfies analogous conditions of Hyp. \ref{hypSet2}, one will probably have to consider  a $2^d\times 2^d$ matrix in the decomposition, where $d$ is the number of primes lying above $p$. Another natural question one may ponder is to what extent can one relax the conditions in Hyp. \ref{hypSet1}, studying representations that are further from being ordinary at $p$. {The results on $\GL_2\times\GL_2$ in \cite{BLLV} suggest that further generalizations may be possible when the local representation can be described using smaller 1-dimensional and 2-dimensional representations upon constructing a larger matrix.}  We plan to study these questions in the future.

\textbf{Signed Iwasawa main conjectures for automorphic representations.} The explicit description of the local representation $V_\Pi$ given by {Proposition}~\ref{MT}  allows us to define signed Coleman maps over $\Qp(\mu_{p^\infty})$, generalizing previous works on {elliptic }modular forms and elliptic curves (e.g.  \cite{kobayashi03,sprung09,lei11compositio,leiloefflerzerbes10}). We utilize the kernels of these Coleman maps as local conditions at $p$ to define signed Selmer groups for $\Pi$ over $\QQ(\mu_{p^\infty})$ (see Definition~\ref{defn:signedSel}) and formulate Iwasawa main conjectures connecting the Pontryagin duals of the signed Selmer groups with the bounded $p$-adic $L$-functions $\sL_p^\#,\sL_p^\flat$ given by Theorem~\ref{MT2} (see Conjecture \ref{conj:IMC}).

The structure of this article is as follows. In Section~\ref{prep}, we introduce  notation and collect preliminary results that will be used in the rest of the article. We study in Section \ref{structure},  the structure of the local Galois representation $V_\Pi$ and prove  {Proposition}~\ref{MT}. We then make use of {Proposition}~\ref{MT} to define bounded signed Coleman maps and study  their interpolation properties in Section~\ref{structure}. Once this is done, we construct the bounded $p$-adic $L$-functions  $\sL_p^\#,\sL_p^\flat$ given in Theorem~\ref{MT2} and explain how to recast the work of Rockwood in \cite{rob} in terms of the $p$-adic $L$-functions we construct.  Section~\ref{Selmer} deals with the construction of signed Selmer groups which are conjecturally cotorsion modules over the cyclotomic Iwasawa algebra.  We finish the article with a discussion on the signed Iwasawa main conjectures and how they are related to the weak Leopoldt conjecture and Perrin-Riou's conjecture on the existence of Euler systems for $\Pi$.

\section*{Acknowledgements}
The authors would like to thank Antonio Cauchi, Rob Rockwood  and Chris Williams for answering many of our questions.    The second named author is indebted to Laurent Clozel, Eknath Ghate, Dipendra Prasad, A. Raghuram and Sujatha Ramdorai for very interesting discussions during the preparation of this article. The first named author's research is supported by the  NSERC Discovery Grants Program RGPIN-2020-04259 and RGPAS-2020-00096. 

Parts of this work were carried out during the thematic semester ``Number Theory -- Cohomology in Arithmetic" at Centre de recherches mathématiques (CRM) in fall 2020. Both authors thank the CRM for the hospitality and very generous supports. The second named author also received fundings from TIFR and the Institute for Advancing Intelligence (IAI), TCG CREST during the revision of the article. Finally, we thank the anonymous referees for very valuable comments and suggestions on  earlier versions of the article.

\section{Preliminaries}\label{prep}
In this section, we review results on automorphic representations and local global compatibility. We also recall Perrin-Riou maps for crystalline representations. This will lay the ground work for the construction of signed Coleman maps and other related objects in subsequent sections.
\subsection{RAESDC automorphic representations}\label{sec:RASEDC}
We say that $\Pi$ is an RAESDC ({regular algebraic, essentially self-dual, cuspidal}) representation of $\GL_{2n}(\A_\Q)$ if $\Pi$ is a cuspidal automorphic representation such that:
\begin{enumerate}
    \item The contragredient $\Pi^{\vee}$ of $\Pi$ satisfies $\Pi^{\vee} \cong \Pi \otimes \chi$ for some Hecke character $\chi: \A_\Q^{\times}/\Q^{\times} \rightarrow \C^\times$;
    \item Writing $\Pi=\Pi_\infty \otimes \Pi_f$, the infinite part $\Pi_\infty$ has the same infinitesimal character as some irreducible algebraic representation of $\GL_{2n}(\R)$.
\end{enumerate}
Let $\mu=(\mu_1,...,\mu_{{2n}}) \in \Z^{2n}$ satisfy the dominant weight condition $$\mu_1 \geq \cdots \geq \mu_{2n}.$$ 
Let $E_\mu$ be the irreducible algebraic representation of $\GL_{2n}(\R)$ with highest weight $\mu$. We say that a RAESDC automorphic representation has weight $\mu$ if $\Pi_\infty$ has the same infinitesimal character as  $E_\mu^\vee$. Let $\mathfrak{g}_\infty=\Lie(\GL_{2n}(\R))$ and $K_\infty$ be the product of a maximal compact subgroup of the real Lie group $\GL_{2n}(\R)$ with the center of $\GL_{2n}(\R)$.
The algebraic regularity condition in  (2) above holds if and only if  $\Pi_\infty$ is \textit{cohomological}, i.e.  $$H^q(\mathfrak{g}_\infty,K_\infty;\Pi_\infty \otimes E_\mu^{}) \neq 0$$
for some degree $q$. Here $H^q(\mathfrak{g}_\infty,K_\infty;\Pi_\infty \otimes E_\mu^{})$ is the space of $(\mathfrak{g}_\infty, K_\infty)$-cohomology in degree $q$ (see \cite[\S I.5]{BorelWallach}). In this case, $\mu$ also satisfies the purity condition
\begin{equation}\label{purity}
\mu_i+\mu_{2n+1-i}=w, \qquad \text{for } i = 1,...,n \text{ and some } w \in \Z.
\end{equation}
Henceforth we assume that $\Pi$ is the transfer of a globally generic cuspidal automorphic representation of $\mathrm{GSpin}_{2n+1}(\A_\Q)$. This functorial transfer has been
established for unitary globally generic cuspidal automorphic representations by Asgari--Shahidi in \cite[Theorem~1.1]{AsgariShahidi1} in its weak form and in \cite[Corollary~4.25]{AsgariShahidi2} at every place. For non-unitary representations $\Pi$, this transfer is discussed in  \cite[p. 686]{GrobnerRaghuram}. Moreover, the transfer of a globally generic cuspidal automorphic representation of $\mathrm{GSpin}_{2n+1}(\A_\Q)$ admits a global Shalika model in the sense of Grobner--Raghuram and are essentially self dual (see \cite[Proposition~3.14]{GrobnerRaghuram}).
\subsection{Strict local-global compatibility}

Let $\rho: G_\Q \rightarrow \GL_{2n}(\overline{\Q}_\ell)$ be a  Galois representation of the absolute Galois group $G_\Q$ of $\Q$.  Assume that $\rho$ is geometric, that is, it is unramified outside a finite set of primes of $\Q$ and its restrictions to the decomposition groups at primes above $\ell$ are potentially semistable in the sense of Fontaine (see for example \cite[\S6.3]{FontaineOuyang}). Let $WD_p$ be the Weil--Deligne group at $p$. For a  geometric representation, one can define a Weil--Deligne representation $WD_p \rightarrow \GL_{2n}(\overline{\Q}_\ell)$ up to conjugacy. This definition is classical for {$p \neq \ell$} and comes from  Deligne--Grothendieck  whereas for {$p=\ell$}, it is due to Fontaine. {A concise survey of both  constructions with references are given in \cite[p. 77-79]{Taylor1}}. 
\begin{defn}\cite[p. 81--82]{Taylor1}, \cite[Section 2]{GhateKumarPublished}
A $\Q$-rational, \textit{strictly compatible} (or \textit{strongly compatible}) system of geometric representations {$(\rho_\ell)$} of $G_\Q$ is a collection of data consisting of: 
\begin{enumerate}
    \item For each prime $\ell$ and each embedding $i: \Q \hookrightarrow \overline{\Q}_\ell$, a continuous, semisimple representation $\rho_\ell:G_\Q \rightarrow \GL_{2n}(\overline{\Q}_\ell)$ that is geometric.
    \item For each prime $p$ of $\Q$, a Frobenius semisimple representation $r_p: WD_p \rightarrow \GL_{2n}(L)$ such that 
    \begin{itemize}
        \item $r_p$ is unramified for all $p$ outside a finite set.
        \item For each $\ell$, the Frobenius semisimple Weil-Deligne representation $WD_p \rightarrow \GL_{2n}(\overline{\Q}_\ell)$ associated to {$\rho_\ell|_{G_{\Q_p}}$}  is conjugate to $r_p$ via the embedding $i: \Q \hookrightarrow \overline{\Q}_\ell$.
        \item There exists a multiset of integers $H$ such that for each prime {$\ell$} and each embedding $i: \Q \hookrightarrow \overline{\Q}_\ell$, the the multiset of Hodge--Tate weights of $\rho_\ell|_{G_{\Q_\ell}}$ is $H$.
    \end{itemize}
\end{enumerate}
\end{defn}
The following conjecture is given in \cite[Conjecture~3.4]{Taylor1} (see also \cite[Conjecture~3.3]{GhateKumarPublished}).
\begin{conjecture}\label{conj:Taylor}
Suppose $\Pi$ is a cuspidal automorphic form on $\GL_{2n}(\A_\Q)$ with infinitesimal character $\chi_H$ where $H$ is a multiset of distinct integers. Then there is a strictly compatible system of Galois representations $(\rho_{\Pi,\ell})$ associated to $\Pi$ with Hodge--Tate weights $H$ such that the local-global compatibility holds for all primes.
\end{conjecture}
Here, local-global compatibility means that the underlying {Frobenius  semi-simplified} Weil--Deligne representation at $p$ in the compatible system (which is independent of the residue characteristic $\ell$ of the coefficients by hypothesis) corresponds to $\Pi_p$ via the local Langlands correspondence. 

 There is  a significant evidence towards this conjecture for RASEDC representations, thanks to the works of Clozel,  Harris, Taylor for {$p \neq \ell$} \cite{ClozelHarrisTaylor},  and   Barnet-Lamb, Gee, Geraghty and Taylor for {$p=\ell$} \cite{LambGeeGeraghtyTaylorII}.

\begin{conv}
We normalize the Hodge--Tate weights on $G_{\Qp}$-representations so that the cyclotomic character  has Hodge--Tate weight $+1$.
\end{conv}

\begin{theorem}\label{Thm:automorphy}
Let $\iota: \overline{\Q}_\ell \cong \C$.
Suppose $\Pi$ is an RASEDC automorphic representation of $\GL_{2n}(\A_\Q)$ of weight $\mu$ as in Section \ref{sec:RASEDC}. Then there is a number field $\mathrm{E}(\Pi)$ and a compatible system $\rho_{\Pi,\lambda, \iota}: G_\Q \rightarrow \GL_{2n}(\mathrm{E}(\Pi)_\lambda)$ of continuous $\lambda$-adic representations where $\lambda$ runs through finite places of $\mathrm{E}(\Pi)$, such that 
\begin{enumerate}
    \item If $p$ is coprime to the prime $\ell$ dividing the norm $N_{\mathrm{E}(\Pi)/\mathbb{Q}}(\lambda)$, we have $${WD(\rho_{\Pi,\lambda, \iota}|_{G_{\mathbb{Q}_p}})^{\mathrm{ss}}}=\Big(r_\ell\big(\Pi_p\circ \iota\big)^{\vee}(1-2n)\Big)^{\mathrm{ss}}$$
    where $r_\ell$ is the reciprocity map defined in \cite{HarrisTaylorGeometry}.
    \item If $p$ divides $N_{\mathrm{E}(\Pi)/\Q}(\lambda)$, and if $\Pi_p$ is unramified, the representation $\rho_{\Pi,\lambda, \iota}|_{G_{\Q_p}}$ is crystalline with Hodge--Tate weights $-h_i=-(\mu_i+2n-i)$ for $i=1,...,2n$ and each of these Hodge--Tate weights have multiplicity one. (the minus signs arise since the weights are the negatives of the jumps in the Hodge filtration on the associated filtered $\varphi$-module $\Dcris(\rho_{\Pi,\lambda, \iota}|_{G_{\Q_p}})$ constructed by Fontaine, see \cite[Definition 6.29]{FontaineOuyang}). {Furthermore, 
      the characteristic polynomial of  $\varphi$  equals the characteristic polynomial of the geometric Frobenius at $p$ of the Weil--Deligne representation  $r_p\big(\Pi_p\circ \iota\big)^{\vee}(1-2n)$.}
    \item The local-global compatibility holds at primes  $p$ dividing $N_{\mathrm{E}(\Pi)/\Q}(\lambda)$. That is, $$\iota WD(\rho_{\Pi,\lambda, \iota}|_{G_{\mathbb{Q}_p}})^{\mathrm{ss}}=\mathrm{rec}(\Pi_p \otimes |\det|^{\frac{1-2n}{2}})^{\mathrm{ss}},$$
    where $WD(\rho_{\Pi,\lambda, \iota}|_{G_{\mathbb{Q}_p}})$ is the Weil--Deligne representation associated to $\rho_{\Pi,\lambda, \iota}|_{G_{\mathbb{Q}_p}}$ and $\mathrm{rec}$ is the local Langlands correspondence \cite{HarrisTaylorGeometry}.
\end{enumerate}
\end{theorem}
\begin{proof}
For (1) and (2), see Chenevier--Harris \cite[Theorem~4.2]{ChenevierHarris} and Geraghty \cite[Proposition~2.27]{Geraghty}. See also \cite[Theorem~3.5]{GhateKumarPublished}. The case of $p$ being coprime to  $N_{\mathrm{E}(\Pi)/\mathbb{Q}}(\lambda)$ is  due to Clozel--Harris--Taylor \cite[Proposition~4.3.1]{ClozelHarrisTaylor}.
For (3), see  \cite[Theorem~A]{LambGeeGeraghtyTaylorII}.
\end{proof}
\begin{remark}\label{conjugation}
As $2n$ is even, we note that $\mathrm{trace} (\rho_{\Pi,\lambda, \iota}(c))=0$, where $c$ is the complex conjugation (see \cite[Theorem~1.1]{CaraianiHung}).  
\end{remark}
\subsection{The Perrin-Riou map for local representations}\label{PRmapsprep}
 Let $\Gamma=\Gal(\QQ(\mu_{p^\infty})/\QQ)\cong\Delta\times\Gamma_1$, where $\Delta\cong\ZZ/(p-1)\ZZ$ and $\Gamma_1\cong\Zp$. We choose a topological generator $\gamma$ of $\Gamma_1$. We fix a finite extension $F$ of $\Qp$, whose ring of integers is denoted by $\cO$.

We write $\Lambda(\Gamma)=\cO\lb \Gamma\rb$ and  $\Lambda(\Gamma_1)=\cO\lb\Gamma_1\rb$ for the Iwasawa algebra of $\Gamma$ and $\Gamma_1$ over $\cO$ respectively. We consider $\Lambda(\Gamma)$ and $\Lambda(\Gamma_1)$ as  subrings of  $\cH(\Gamma)$ and $\cH(\Gamma_1)$, which are  the rings of power series $f \in F[\Delta]\lb X \rb$ (respectively $f \in F\lb X \rb$) which converge on the open unit disc $|X| < 1$ {in $\Cp$, where $|\ \ |$ denotes the $p$-adic norm on $\Cp$ normalized by $|p|=p^{-1}$}. For any real number $r\ge 0$, we write  $\cH_{r}(\Gamma)$ and $\cH_{r}(\Gamma_1)$ for the set of power series $f$ in $\cH(\Gamma)$ and $\cH(\Gamma_1)$ respectively satisfying $\sup_{t} p^{-tr}\Vert f\Vert_{\rho_t}<\infty$,
where $\rho_t=p^{-1/p^{t-1}(p-1)}$ {and $\Vert f\Vert_{\rho_t}=\sup_{|z|\le \rho_t}|f(z)|$}. It is common to write $f=O(\log_p^r)$ when $f$ satisfies this condition.

 Given an integer $i$, we write $\Tw^i$ for the $\Qp$-algebra automorphism of $\cH(\Gamma)$ defined by $\sigma \mapsto \cc^i(\sigma)\sigma$ for $\sigma\in \Gamma$.
We set $u:=\cc(\gamma)$. If $m\ge 1$ is an integer, we define $\Phi_m$ to be the $p^m$-th cyclotomic polynomial in $1+X$, namely $\displaystyle\frac{(1+X)^{p^m}-1}{(1+X)^{p^{m-1}}-1}$.
We let $\log_p=\log_p(1+X) \in \cH(\Gamma_1)$ denote the $p$-adic logarithm.
We also define for an integer $i$, the element 
\begin{equation}\label{eq:ell}
   \ell_i=\frac{\log_p}{\log_p(u)}-i=\frac{\Tw^{-i}\left(\log_p\right)}{\log_p(u)}\in \cH(\Gamma_1) 
\end{equation}
For $i\ge0$, we define the product $$\tell_i=\prod_{j=0}^{i-1}\ell_j.$$
Note that for $i=0$, we have $\tell_0=1$.

\begin{defn}\label{defn:basic-objects}
Let $T$ be a $G_{\Qp}$-stable $\cO$-lattice of a finite dimensional $F$-linear crystalline representation $V$ of $G_{\Q_p}$ with non-negative Hodge--Tate weights such that $V$ has no sub-quotient isomorphic to the trivial representation $F$. 
\item[i)] We define $\HIw(\Qp,T)$ to be the inverse limit $\varprojlim H^1(\Qp(\mu_{p^m}),T)$, where the connecting maps are corestrictions.% We also define $\HIw(\Qp,V)$ to be $\HIw(\Qp,T)\otimes\Qp$.
\item[ii)] We write $\NN(T)$ for the Wach module of $T$ (see for example \cite[\S II.1]{berger04} for the precise definitions; it is a filtered module over the ring $\cO[[\pi]]$, where $\pi$ can be regarded as a formal variable  equipped with an action of $\vp$ and $\Gamma$ given by $\vp(\pi)=(1+\pi)^p-1$ and $\sigma(\pi)=(1+\pi)^{\cc(\sigma)}-1$ for $\sigma\in\Gamma$)\footnote{{We recall that $\pi$ is in the ring of of Witt vectors of $\varprojlim_{x\mapsto x^p}\Cp$ given by $[(1, \zeta_p,\zeta_{p^2},...)]-1$, where $\zeta_{p^n}$ is a primitive $p^n$-th root of unity in $\Cp$ such that $\zeta_{p^{n+1}}^p=\zeta_{p^n}$. It is used to define Fontaine's period $t=\log(1+[\pi])\in \mathbb{B}_\mathrm{dR}^+$.}}. We have the integral Dieudonn\'e module $\Dcris(T)$ given by $\NN(T)$ modulo $\pi$.
\item[iii)]We write $\cL_T:\HIw(\Qp,T)\rightarrow \cH(\Gamma)\otimes\Dcris(T)$ %and $\cL_V:\HIw(\Qp,V)\rightarrow \cH(\Gamma)\otimes\Dcris(V)$
for the Perrin-Riou map as defined in \cite[\S3.1]{leiloefflerzerbes11} and \cite[Appendix B]{LZ0} {(see also our discussion below for a review of this map)}.
\end{defn}

 Let $\psi$ dentoe the left inverse of $\vp$ as given in \cite[\S I.2]{berger03}. We write $\Brig$ for the ring of power series in $F[[\pi]]$ that converge on the open unit disk. Let  $t=\log_p(1+\pi)\in \Brig$ and $q=\vp(\pi)/\pi$. We recall that the Mellin transform $\mathfrak{m}$ sending $1$ to $1+\pi$ induces the $\Lambda(\Gamma)$-isomorphisms
\[
\Lambda(\Gamma)\stackrel{\sim}{\longrightarrow}\cO[[\pi]]^{\psi=0},\qquad\cH(\Gamma)\stackrel{\sim}{\longrightarrow}\left(\Brig\right)^{\psi=0}.
\]

{Note that the condition $\psi=0$ simply cuts out the distributions supported on $\Gamma=\Z_p^\times \subset \Z_p$ which is where these natural Mellin transform isomorphisms come from.}

We fix a $\Zp$ basis $e_1$ for the $G_{\Qp}$-representation $\Zp(1)$. Given an integer $k$, we write $e_k=e_1^{\otimes k}$. We have the identifications
\begin{align*}
    \HIw(\Qp,T(k))&=\HIw(\Qp,T)\cdot e_k,\\
    \Dcris(T(k))&=\Dcris(T)\cdot t^{-k} e_k,\\
    \NN(T(k))&=\NN(T)\cdot \pi^{-k}e_k.
    \end{align*}

We recall from \cite[Theorem~A.3]{berger03} and \cite[\S1.3]{BB} that the assumption that the  Hodge--Tate weights are non-negative  implies that there is an isomorphism of $\Lambda(\Gamma)$-modules
\[
    \HIw(\Qp,T)\cong \NN(T)^{\psi=1}
\]
{via the Herr complex. Here, the superscript $\psi=1$ signifies the kernel of the morphism $\psi-1$.}
The Perrin-Riou map {$\cL_T:\HIw(\Qp,T)\rightarrow \cH(\Gamma)\otimes\Dcris(T)$ in Definition~\ref{defn:basic-objects}} can be defined via this isomorphism composed with
\[
\NN(T)^{\psi=1}\stackrel{1-\vp}{\longrightarrow}(\Brig)^{\psi=0}\otimes\Dcris(T)\stackrel{\mathfrak{m}^{-1}\otimes 1}{\longrightarrow}\cH(\Gamma)\otimes\Dcris(T)
\]
{(see \cite[Proposition 1.1]{leiloefflerzerbes10}).}
We will also need the following notations while constructing Selmer groups in Section \ref{Selmer}.

\begin{enumerate}
\item[i)]Given a $\cO$-module $M$, we  write $M^\vee=\Hom_{\mathrm{cts}}(M,F/\cO)$ for its Pontryagin dual.
\item[ii)]For a finitely generated torsion $\Lambda(\Gamma_1)$-module $M$,  let $\ch_{\Lambda(\Gamma_1)}(M)$ denote the characteristic ideal of $M$. 
\item[iii)]For a Dirichlet character $\eta: \Delta \rightarrow \Z_p^\times$, we define $e_\eta$ to be the idempotent attached to $\eta$ given by $\frac{1}{p-1}\sum_{\sigma \in \Delta}\eta(\sigma)^{-1}\sigma \in \Z_p[\Delta].$
\item[iv)]Given a $\Lambda(\Gamma)$-module $M$ and a character $\eta$ as above, we define its $\eta$-isotypic component to be
\[
M^\eta=e_\eta\cdot M,
\]
which we consider as a $\Lambda(\Gamma_1)$-module.
\item[v)]Given an element $f=\sum_{k \geq 0, \sigma \in \Delta}a_{k, \sigma}\cdot \sigma \cdot (\gamma -1)^k \in \cH(\Gamma)$, where $a_{k, \sigma} \in F$ and  $\gamma$ is a topological generator of $\Gamma_1$, we may decompose $f$ into $\sum_{\eta\in\hat\Delta} f^\eta$, where  $f^\eta=e_\eta\cdot f\in\cH(\Gamma)^\eta$. Furthermore, we may identify $f^\eta$  with an element of $\cH(\Gamma_1)$ given by $\sum_{k\geq 0}\Big( \sum_{\sigma \in \Delta}a_{k, \sigma}\eta(\sigma) \Big) (\gamma -1)^k$.
 \end{enumerate}

\section{Structure of non-ordinary local Galois representation at $p$}\label{structure}
%The main goal of this section is to prove  Theorem \ref{MT}.

Fix an integer $n \geq 1$ and set $G = \GL_{2n}$. Let $T$ be the standard diagonal split torus of $G$ and $B$ the upper triangular Borel subgroup.
Now suppose $\Pi$ is a RAESDC automorphic representation of $\GL_{2n}(\A_\Q)$ which is the transfer of a globally generic cuspidal automorphic representation of $\mathrm{GSpin}_{2n+1}(\A_\Q)$ and suppose $\Pi_p$ is unramified. Let us suppose, there is an unramified character $$\chi_p: T(\mathbb{Q}_p) \rightarrow \mathbb{C}^\times$$ such that $$\Pi_p=\mathrm{Ind}_{{B}(\mathbb{Q}_p)}^{G(\mathbb{Q}_p)}(|\cdot|^{\frac{2n-1}{2}}\chi_p).$$
We define the Satake parameters at $p$ to be the values $\alpha_i=\chi_{p,i}(p)$, where the $\chi_{p,i}$'s
denote the projections to the diagonal entries. After choosing an isomorphism $\Q_p \cong \C$, the indices $i$ are ordered so that $\ord_p(\alpha_1)\ge \ord_p(\alpha_2)\ge\cdots \ge \ord_p(\alpha_{2n})$. For all $i$, we also have 
\begin{equation} \label{alphaParity}
\alpha_i \alpha_{2n+1-i}=\lambda
\end{equation}
for fixed $\lambda$ with $p$-adic valuation $2n-1{+w}$. As in the introduction, we take $F=E(\Pi)_\lambda$, where $\lambda$ is a prime of $E(\Pi)$ above $p$. By extending scalars if necessary, we assume that all $\alpha_i$ are contained in $F$.
 %Then we have the following lemma in the $(p,p)$ case.
\begin{lemma}\label{lem:phi}
When $p=\ell$, the eigenvalues of $\varphi$ on $\Dcris(\rho_{\Pi,\lambda, \iota}|_{G_{\Q_p}})$ are given by $\alpha_i$ for $i=1,...,2n$. 
\end{lemma}
\begin{proof}
{This is a direct consequence of Theorem \ref{Thm:automorphy}}.\end{proof}
%Since $\Pi_p$ is an unramified principal series, by \cite[Lem. 3.1.1]{ClozelHarrisTaylor}, we know that $\Pi_p$ has a nonzero vector fixed by the maximal compact subgroup $\GL_{2n}(\Z_p)$ and therefore satisfies the hypothesis in part (2) of Theorem \ref{Thm:automorphy}. Therefore, the eigenvalues of $\varphi$ are precisely the eigenvalues of the geometric Frobenius at $p$ of $r_p\big(\Pi_p\circ \iota\big)^{\vee}(1-2n)$. 

%The characteristic polynomial of the geometric Frobenius at $p$ of $r_p\big(\Pi_p\circ \iota\big)^{\vee}(1-2n)$ is given by $$\prod_{i=1}^{2n}(X-|p|^{\frac{2n-1}{2}}\chi_p(p)p^{\frac{2n-1}{2}})=\prod_{i=1}^{2n}(X-\alpha_i)$$ (see the discussion in  \cite[p.393]{GhateKumarPublished}).

%This completes the proof of the lemma.

 Henceforth, we shall write $V_\Pi$ for the representation {$\rho_{\Pi,\lambda,\iota}^*$. By an abuse of notation, we may write $V_\Pi$ to denote $V_\Pi|_{G_{\Qp}}$ when no confusion arises.} We recall from Theorem~\ref{Thm:automorphy} that the jumps in the Hodge filtration of $V_\Pi$ are  $-(\mu_i+2n-i)$. {Therefore,} the $\vp$-eigenvalues on $\Dcris(V_\Pi)$ are given by $\alpha_i^{-1}$. 

%\subsection{Setup} \label{sec:setup}
We work under the following hypotheses.
\begin{hyp}\label{hypSet1}
We assume that the Hodge--Tate weights {and the Satake parameters of} $\Pi$ at $p$ satisfy the following hypotheses.

\begin{itemize}
\item[\normalfont(M.Slo)]$\ord_p(\alpha_i)=h_i,\ i=n+2,\ldots, 2n$;
\item[\normalfont(N.ord)]{$\ord_p(\alpha_{n+1})>h_{n+1}$} and $\alpha_n\ne \alpha_{n+1}$;
    \item[\normalfont(G.Po)] The Newton filtration on $\Dcris(V_\Pi)$ is in general position with respect to the Hodge filtration.
\end{itemize}
\end{hyp}
%Let us first briefly discuss the significations of these hypotheses.

The hypothesis (M.Slo) is the ``minimal slope"  hypothesis in \cite[\S3]{rob}. As the Newton polygon of $\Pi$ lies on or above the Hodge polygon of $\Pi$  and their end points coincide (see Hida's work \cite[Section 8.2]{Hida1}), it is easy to see from \eqref{purity} and \eqref{alphaParity} that  (M.Slo) implies 
\begin{equation}\label{lem:slope}
\ord_p(\alpha_i)=h_i \text{ for } i=1,\ldots ,n-1.
\end{equation}
Hence the Newton and the Hodge polygons coincide at all points except the $n$-th point. 

{Note that the inequality  $\ord_p(\alpha_{n+1})>h_{n+1}$ in  (N.ord) implies that  $$h_n>\ord_p(\alpha_n),\ord_p(\alpha_{n+1})>h_{n+1}$$ (once again thanks to \eqref{purity} and \eqref{alphaParity}). We remark} that the hypothesis (N.ord) is a weaker condition than the Pollack condition $\alpha_n+\alpha_{n+1}=0$ considered in  \cite{rob}. Indeed, the Pollack condition  has the consequence that $\ord_p(\alpha_n)=\ord_p(\alpha_{n+1})=\frac{1}{2}(h_n+h_{n+1})$, which is a special case of {$\ord_p(\alpha_{n+1})>h_{n+1}$}.

%Our hypotheses wilthat l tell us  that certain 2-dimensional sub-quotient of $V_\Pi$ is a non-ordinary representation (see  Theorem~\ref{thm:upper-tri} below).

Hypothesis (G.Po) is \cite[Assumption~3.6]{GhateKumarPublished}, which means that given a $d$-dimensional $\vp$-stable subspace of $\Dcris(V_\Pi)$, the jumps in its Hodge filtration occur at the first  $d$  jumps in the Hodge filtration of $\Dcris(V_\Pi)$. As discussed in op. cit., (G.Po) is expected to hold generically.

We are now ready to prove {Proposition}~\ref{MT}.

\begin{proposition}\label{thm:upper-tri}
Assume that Hyp \ref{hypSet1} holds. The local representation ${V_\Pi|_{G_{\Qp}}}$ is isomorphic to a representation of the form
\[
\begin{pmatrix}
\cc^{h_{1}}\theta_{1} &*&\cdots&*&*&*&*&\cdots&*&*\\
0&\cc^{h_{2}}\theta_{2} &\cdots&*&*&*&*&\cdots&*&*\\
\vdots&\ddots&\ddots&&\vdots&\vdots&\vdots&&\vdots&\vdots\\
0&\cdots&0&\cc^{h_{n-1}}\theta_{n-1}&*&*&*&\cdots&*&*\\
0&\cdots&\cdots&0 &\times&\times&*&\cdots&*&*\\
0&\cdots&\cdots&0 &\times&\times&*&\cdots&*&*\\
0&\cdots&\cdots&0 &0&0&\cc^{h_{n+2}}\theta_{n+2}&\cdots&*&*\\
\vdots&&& \vdots&\vdots&\vdots&\ddots&\ddots&\vdots&\vdots\\
0&\cdots&\cdots&0 &0&0&\cdots&0&\cc^{h_{2n-1}}\theta_{2n-1}&*\\
0&\cdots&\cdots&0 &0&0&\cdots&\cdots&0&\cc^{h_{2n}}\theta_{2n}
\end{pmatrix},
\]
where $\cc$ denotes the $p$-adic cyclotomic character, $\theta_i$ are unramified characters on $G_{\Qp}$ and the middle $2\times 2$ square, given by the symbol $\times$, is a two-dimensional sub-quotient $V_\Pi'$ of $V_\Pi$ whose Hodge--Tate weights are $h_n$ and $h_{n+1}$. Furthermore, the $\vp$-eigenvalues on $\Dcris(V_\Pi')$ are $\alpha_{n}^{-1}$ and $\alpha_{n+1}^{-1}$.
\end{proposition}
\begin{proof}
%\JR{This proposition is essentially a consequence of \cite{LambGeeGeraghtyTaylorII} under the hypothesis Hyp \ref{hypSet1}. However, we include a detailed proof for a non-expert reader.}
Given that $h_i=\mu_i+2n-i$ and $\mu_1\ge\cdots\ge\mu_{2n} $, we have
\[
h_1>\cdots >h_{2n}.
\]
This, together with (\ref{lem:slope}) and (M.Slo) imply that
\[
\ord_p(\alpha_1)>\cdots >\ord_p(\alpha_{n-1})>\ord_p(\alpha_{n}),\ord_p(\alpha_{n+1})>\ord_p(\alpha_{n+2})>\cdots > \ord_p(\alpha_{2n}).
\]
Therefore, the characteristic polynomial of $\vp|_{\Dcris(V_\Pi)}$ factors 
into the product
\[
Q(X)\prod_{\substack{1\le i \le n-1\\ n+2\le i\le 2n}}(X-\alpha_i^{-1}),
\]
where $Q(X)$ is a monic quadratic polynomial defined over $F$ whose roots  are $\alpha_n^{-1}$ and $\alpha_{n+1}^{-1}$. This tells us that 
for $1\le i\le n-1$ and $n+2\le i\le 2n$,  the $\vp$-eigensubspace of $\Dcris(V_\Pi)$ for the eigenvalue $\alpha_i^{-1}$, which we denote by $E_i$, is a one-dimensional $F$-subspace of $\Dcris(V_\Pi)$. Let $E'$ denote the subspace of $\Dcris(V_\Pi)$ given by the kernel of $Q(\vp)$. Since the $\alpha_i$'s are all distinct, we have
\[
\Dcris(V_\Pi)=E'\oplus\bigoplus_{\substack{1\le i \le n-1\\ n+2\le i\le 2n}}E_i.
\]

We define the following subspaces of $\Dcris(V_\Pi)$:
\begin{align*}
D_i&=\bigoplus_{j=1}^{i}E_j,\qquad 1\le i\le n-1,\\
D_{n}&=D_{n+1}=D_{n-1}\oplus E',\\
D_i&=D_n\oplus \bigoplus_{j=n+2}^{i}E_j,\quad n+2\le i\le 2n.
\end{align*}
This gives the filtration
\[
0=:D_{0}\subsetneq D_{1}\subsetneq\cdots \subsetneq D_{n-1}\subsetneq D_{n}=D_{n+1}\subsetneq D_{n+2}\subsetneq\cdots \subsetneq D_{2n-1}\subsetneq D_{2n}=\Dcris(V_\Pi).
\]
Then, by the hypothesis (G.Po) and (\ref{lem:slope}), for all $1\le i\le 2n$, $D_i$ is an admissible filtered $(\vp,N)$-module with $t_N(D_i)=t_H(D_i)$ (we refer the reader to \cite[\S2]{GhateKumarPublished} for the notation and terminology being used here; note that  we are taking the fields $E$ and $F$ in op. cit. to be $F$ and $\Qp$ here). Consequently, $t_N(D_i/D_{i-1})=t_H(D_i/D_{i-1})$. The equivalence of categories between admissible filtered  $(\vp,N)$-modules and $G_{\Qp}$-representations proved by Colmez--Fontaine \cite{CF} (see also \cite[Theorem~2.4]{GhateKumarPublished}) tells us that for each $i$,  $D_i=\Dcris(V_i)$ for some $G_{\Qp}$-sub-representation $V_i$ of $V_\Pi$ such that 
\[
0=V_{0}\subsetneq V_{1}\subsetneq\cdots \subsetneq V_{n-1}\subsetneq V_n=V_{n+1}\subsetneq V_{n+2}\subsetneq\cdots \subsetneq V_{2n-1}\subsetneq V_{2n}=V_\Pi,
\]
where $V_i/V_{i-1}$ one-dimensional with Hodge--Tate weight $h_i$ for $i=1,2,\ldots, n-1,n+2,\ldots, 2n$ and $V_{n}/V_{n-1}$ is 2-dimensional with Hodge--Tate weights $h_n$ and $h_{n+1}$. This finishes the proof of the {proposition}.
\end{proof} 

\begin{remark}
Note that when $n=2$, {the representations studied in {Proposition}~\ref{thm:upper-tri} have exactly the same form as those} considered in \cite[Corollary~1(i)]{urban05}.
\end{remark}

\section{Construction of bounded Coleman maps and $p$-adic $L$-functions}\label{sec:PR_maps}
We study $\Lambda(\Gamma)$-valued Coleman maps and {a} certain logarithmic matrix attached to $T_\Pi$. This allows us to prove Theorem~\ref{MT2}.

\subsection{Perrin-Riou maps and Coleman maps}\label{sec:col}
Let  $V_\Pi$ be as in previous sections. Recall that $T_\Pi$ is a $G_\QQ$-stable lattice of $V_\Pi$. We  continue to assume that Hyp \ref{hypSet1} holds. The main goal of this section is to define bounded Coleman maps, decomposing the Perrin-Riou map attached to $T_\Pi$.
 If $V_i$ is one of the {sub-representations} in the proof of {Proposition} ~\ref{thm:upper-tri}, we write $T_i$ to be the  $\cO$-lattice $V_i\cap T_\Pi$ inside $V_i$. 
Our construction of Coleman maps relies on patching together the Coleman maps for  the sub-quotients $T_i/T_{i-1}$, $i=n,\ldots, 2n$.

We work under the following hypotheses. 
\begin{hyp}\label{hypSet2}
From now on, we assume the following additional hypotheses:
\begin{itemize}
    \item[\normalfont (Pos)] $h_{2n}\ge0$ and the characters $\theta_{i}$, $i=n+2,\ldots,2n$ in {Proposition}~\ref{thm:upper-tri} are non-trivial;
\item[\normalfont (FL)]$p>h_n-h_{n+1}>1$. 
\end{itemize}
\end{hyp}

Note that the hypothesis (FL) implies that the two-dimensional representation $V_\Pi'$ given in Theorem~\ref{thm:upper-tri} satisfies the Fontaine--Laffaille condition. {It also ensures that the $p$-adic $L$-functions we consider are non-zero. This condition excludes $p$ from being even. With extra work, it is possible that the case $p=2$ may be treated separately if one does not insist on proving the non-triviality of the $p$-adic $L$-functions.}

Let $T=T_i/T_{j}$, where $i,j\in\{0,1,\ldots,2n\}$ with $j<i$.
The hypothesis (Pos) implies that there is an isomorphism of $\Lambda(\Gamma)$-modules 
\[
    \HIw(\Qp,T)\cong \NN(T)^{\psi=1},
\]
{given by Herr's complex} (see \cite[Theorem~A.3]{berger03} and \cite[\S1.3]{BB}).

\begin{proposition}\label{lem:im1dim}
Let $i\in\{n+2,\ldots,2n\}$. The image of $\cL_{T_i/T_{i-1}}$ {as defined in Definition~\ref{defn:basic-objects}} lands inside $\tell_{h_i}\Lambda(\Gamma)\otimes \Dcris(T_i/T_{i-1})$, {where  $\tell_{h_i}$ is given by \eqref{eq:ell}.} 
\end{proposition}
\begin{proof}
Let us write in this proof $T=T_i/T_{i-1}$. Recall from {Proposition}~\ref{thm:upper-tri} that $T=\cO(\cc^{h_i}\theta_i)$. Since the Hodge--Tate weight of $\cO(\theta_i)$ is 0, it follows that $\cL_{\cO(\theta_i)}$ has image in $\Lambda(\Gamma)\otimes\Dcris(\cO (\theta_i))$.
We recall from \cite[\S4.4]{LZ0} that
\[
\left(\prod_{j=1}^{h_{i}} \ell_{-j}\right)\cL_{\cO(\theta_i)}(z)= (\Tw^{h_{i}}\otimes 1)\left(\cL_{T}(z\otimes e_{h_{i}})\right)\cdot t^{h_{i}}e_{-h_{i}}
\]
for $z\in\HIw(\Qp,\cO(\theta_i))$. A direct calculation shows that
\[
\Tw^{-h_{i}}\left(\prod_{j=1}^{h_{i}} \ell_{-j}\right)=\tell_{h_{i}}.
\]
Therefore, the image of $\cL_T$ lies inside
\[
\tell_{h_{i}}\Lambda(\Gamma)\otimes \Dcris(\cO(\theta_i))\cdot t^{-h_i}e_{h_i}=\tell_{h_{i}}\Lambda(\Gamma)\otimes \Dcris(T)
\]
as required.
\end{proof}

\begin{defn}
For $i\in\{n+2,\ldots,2n\}$, fix an $\cO$-basis  $\omega_i$ of $\Dcris(T_i/T_{i-1})$ (which is necessarily a $\vp$-eigenvector, with $\vp(\omega_i)=\alpha_i^{-1}\omega_i$). Proposition~\ref{lem:im1dim} allows us to define 
\[
\col_{\omega_i}:\HIw(\Qp,T_i/T_{i-1})\rightarrow \Lambda(\Gamma)
\]
to be the unique $\Lambda(\Gamma)$-morphsim satisfying  $\cL_{T_i/T_{i-1}}(z)=\col_{\omega_i}(z)\tell_{h_i}\omega_i$. We also define
\[
\cL_{\omega_i}=\tell_{h_i}\col_{\omega_i}.
\]
\end{defn}
\begin{lemma}\label{rk:injective}
For $i\in\{n+2,\ldots,2n\}$,  the maps $\col_{\omega_i}$ are   injective.
\end{lemma}
\begin{proof}
Note that $\NN(T_i/T_{i-1})^{\vp=1}\subset H^0(\Qp(\mu_{p^\infty}),T_i/T_{i-1})$, which is zero, thanks to our hypothesis that $\theta_i\ne1$ (as given in (Pos)). Therefore,  the argument in \cite[proof of Proposition~4.10]{LZ0} shows that  $\cL_{T_i/T_{i-1}}$ is injective. Consequently, $\col_{\omega_i}$ is also injective.
\end{proof}

We now turn our attention to the two-dimensional sub-quotient $T_{n}/T_{n-1}$. For notational simplicity, we shall write $T_\Pi'$ for $T_{n}/T_{n-1}$. We shall consider the Tate twist $T_\Pi'':=T_\Pi'(-h_{n+1})$. 
\begin{lemma}\label{rk:divisible-log}
The image of $\cL_{T_\Pi'}$ lies inside
$\tell_{h_{n+1}}\cH(\Gamma)\otimes \Dcris(T_\Pi').$
\end{lemma}
\begin{proof}
Note that the Hodge--Tate weights of $T_\Pi''$ are $0$ and $h_n-h_{n+1}\ge0$. The image of $\cL_{T_\Pi''}$ lies inside $\cH(\Gamma)\otimes \Dcris(T_\Pi'')$. Similar to the proof of Proposition~\ref{lem:im1dim}, we have
\begin{equation}
\left(\prod_{i=1}^{h_{n+1}} \ell_{-i}\right)\cL_{T''_\Pi}(z)= (\Tw^{h_{n+1}}\otimes 1)\left(\cL_{T_\Pi'}(z\otimes e_{h_{n+1}})\right)\cdot t^{h_{n+1}}e_{-h_{n+1}}
    \label{eq:PR-twist}
\end{equation}
for $z\in\HIw(\Qp,T'_\Pi)$, which tells us that the image of $\cL_{T_\Pi'}$ lies inside 
$$
\Tw^{-h_{n+1}}\left(\prod_{i=1}^{h_{n+1}} \ell_{-i}\right)\cH(\Gamma)\otimes \Dcris(T_\Pi')=\tell_{h_{n+1}}\cH(\Gamma)\otimes \Dcris(T_\Pi'),
$$
as required.
\end{proof}

Our goal is to define  bounded Coleman maps on $\HIw(\Qp,T_\Pi')$. We do so by first defining such maps on $\HIw(\Qp,T_\Pi'')$ and then twist these maps appropriately.
\begin{defn}\label{defn:twist-col}
\item[i)]We fix an $\cO$-basis $\{\omega_n^\circ,\omega_{n+1}^\circ\}$ of $\Dcris(T_\Pi')$ that is adapted to the filtration of $\Dcris(T_\Pi')$ in the sense that $\omega_n^\circ$ is an $\cO$-basis of $\Fil^{-h_n+1}\Dcris(T_\Pi')$ (see  \cite[\S V.2]{berger04}, where this condition is discussed). We also fix $\vp$-eigenvectors $\omega_n,\omega_{n+1}\in \Dcris(T_\Pi')\otimes_{\cO}F$ with $\vp(\omega_i)=\alpha_i^{-1}\omega_i$ for $i=n$ and $n+1$.
\item[ii)]For $i\in\{n,n+1\}$, we write $\varpi_i^\circ=\omega_i^\circ\cdot  t^{h_{n+1}}e_{-h_{n+1}}\in\Dcris(T_\Pi'')$ and $\varpi_i=\omega_i \cdot t^{h_{n+1}}e_{-h_{n+1}}\in\Dcris(T_\Pi'')\otimes_\cO F$. 
\item[iii)]Let $\{x_n,x_{n+1}\}$ be the Wach module basis of $\NN(T_\Pi'')$ given by \cite[Proposition~V.2.3]{berger04} (which applies to $T_\Pi''$ thanks to our hypotheses (N.ord) and (FL)) lifting $\{\varpi_n^\circ,\varpi_{n+1}^\circ\}$ (meaning that $x_i$ is sent to $\varpi_i^\circ$ under the natural projection  $\NN(T_\Pi'')\rightarrow \Dcris(T_\Pi'')$).
\item[iv)]We define the change of basis matrices $M\in M_{2\times 2}(\Brig)$ and $Q\in M_{2\times 2}(F)$ given by the relations
\[
\begin{pmatrix}
x_n&x_{n+1}
\end{pmatrix}=\begin{pmatrix}
\varpi_n^\circ&\varpi_{n+1}^\circ
\end{pmatrix}M,\quad
\begin{pmatrix}
\varpi_n&\varpi_{n+1}
\end{pmatrix}=\begin{pmatrix}
\varpi_n^\circ&\varpi_{n+1}^\circ
\end{pmatrix}Q.
\]
\item[v)]We define $$\col_{\varpi_n^\circ},\col_{\varpi_{n+1}^\circ}:\HIw(\Qp,T_\Pi'')\rightarrow \Lambda(\Gamma)$$
given by the relation
\[
\cL_{T_\Pi''}(z)=\col_{\varpi_n^\circ}(z)\cdot(1+\pi)\vp(x_n)+\col_{\varpi_{n+1}^\circ}(z)\cdot (1+\pi)\vp(x_{n+1}).
\]
(This makes sense since $M\equiv I_2\mod\pi^{h_n-h_{n+1}}$ and $h_n-h_{n+1}>1$ by (FL). {Thus, \cite[Theorem~3.5]{leiloefflerzerbes10} applies and we may define $\col_{\varpi_n^\circ}$ and $\col_{\varpi_{n+1}^\circ}$ in the same fashion as Definition 3.13 in op. cit.}) 
\item[vi)]We define the logarithmic matrix
\[
M_{\log}''=\mathfrak{m}^{-1}\left((1+\pi)A\vp(M)\right) \in \cH(\Gamma),
\]
where $A\in M_{2\times 2}(F)$ denotes the matrix of $\vp$ with respect to $\{\varpi_n^\circ,\varpi_{n+1}^\circ\}$.

\item[vii)]We also define, for $i=n,n+1$, the $\Lambda(\Gamma)$-morphisms $$\cL_{\varpi_i}:\HIw(\Qp,T_\Pi'')\rightarrow \cH_{\ord_p(\alpha_i)-h_{n+1}}(\Gamma)$$
given by the relation
\[
\cL_{T_\Pi''}(z)=\cL_{\varpi_n}(z)\varpi_n+\cL_{\varpi_{n+1}}(z)\varpi_{n+1}.
\]
\end{defn}

\begin{remark}\label{rk:choice}
 The proof of \cite[Lemma~3.1]{LLZ3} tells us that we may choose $\varpi_{n+1}$ to be $\vp(\varpi_n)$, which we shall do in the rest of the article. In this case, $A$ is given by
 \[
 \begin{pmatrix}
    0&-p^{2h_{n+1}}(\alpha_n\alpha_{n+1})^{-1}\\
    1&p^{h_{n+1}}(\alpha_n^{-1}+\alpha_{n+1}^{-1})
 \end{pmatrix}.
 \]
A direct calculation shows that the matrix $Q$ is of the form
 \[\begin{pmatrix}
    -p^{h_{n+1}}\alpha_{n+1}^{-1}\beta_{n}&-p^{h_{n+1}}\alpha_n^{-1}\beta_{n+1}\\
    \beta_n&\beta_{n+1}
 \end{pmatrix},
 \]
where $\beta_n,\beta_{n+1}\in F^\times$.
\end{remark}

We now define Coleman maps for $T_\Pi'$ by taking appropriate twists of those defined for $T_\Pi''$.

\begin{defn}\label{def:ColBlock}
\item[i)]For $i\in\{n,n+1\}$, we define the Coleman map
\[
\col_{\omega_i^\circ}:\HIw(\Qp,T_\Pi')\rightarrow \Lambda(\Gamma)
\]
as the composition
\[
\HIw(\Qp,T_\Pi')\stackrel{\cdot e_{-h_{n+1}}}{\longrightarrow} \HIw(\Qp,T_\Pi'')\stackrel{\col_{\varpi_i^\circ}}{\longrightarrow}\Lambda(\Gamma)\stackrel{\Tw^{h_{n+1}}}{\longrightarrow}\Lambda(\Gamma).
\]
\item[ii)]We define $\Mlog'$ to be $\Tw^{h_{n+1}}\left(\Mlog''\right)$.

\item[iii)]As in Definition~\ref{defn:twist-col}, we also define, for $i=n,n+1$, the $\Lambda(\Gamma)$-morphisms $$\cL_{\omega_i}:\HIw(\Qp,T_\Pi')\rightarrow \cH_{\ord_p(\alpha_i)}(\Gamma)$$
given by the relation
\[
\cL_{T_\Pi'}(z)=\cL_{\omega_n}(z)\omega_n+\cL_{\omega_{n+1}}(z)\omega_{n+1}.
\]
\end{defn}

\begin{remark}\label{rk:PR2dim}
\item[i)]The maps $\col_{\omega_i^\circ}$ are non-zero and their images can be described explicitly by products of certain linear factors (see \cite[Theorem~5.10]{leiloefflerzerbes11}). 
\item[ii)]The calculations in \cite{BFSuper} (see particularly (9) in op. cit.) show that
\[
\begin{pmatrix}
\cL_{\varpi_n}\\ \cL_{\varpi_{n+1}}
\end{pmatrix}=Q^{-1}M''_{\log}\begin{pmatrix}
\col_{\varpi_n^\circ}\\ \col_{\varpi_{n+1}^\circ}
\end{pmatrix}.
\]
Then \eqref{eq:PR-twist} gives that
\begin{equation}\label{eq:decomp-PR-Col}
  \begin{pmatrix}
\cL_{\omega_n}\\ \cL_{\omega_{n+1}}
\end{pmatrix}=\tell_{h_{n+1}}Q^{-1}M'_{\log}\begin{pmatrix}
\col_{\omega_n^\circ}\\ \col_{\omega_{n+1}^\circ}
\end{pmatrix}.
\end{equation}
\end{remark}

We now define certain projections of the Perrin-Riou map and Coleman maps using wedge products. We expect that the former {will give rise to the $p$-adic $L$-functions of Barrera--Dimitrov--Williams when applied to appropriate cohomology classes, whereas the latter will give rise to certain bounded $p$-adic $L$-functions (see Remark~\ref{rk:ES} for more details).  The latter will also allow us to define signed Selmer groups, which is the content of \S\ref{Selmer} below}.

\begin{defn}\label{defn:PR}
\item[i)]We set
\[
\alpha=\alpha_{n+1}\alpha_{n+2}\cdots \alpha_{2n},\quad \beta=\alpha_n\alpha_{n+2}\cdots \alpha_{2n},
\]
For $\lambda\in\{\alpha,\beta\}$,\footnote{Let $U_p$ be the Hecke operator corresponding to the diagonal matrix $    \left(
    \begin{array}{c|c} 
      p I_n&0 \\ 
      \hline 
      0 & I_n 
    \end{array} 
    \right)$, where $I_n$ denotes the $n\times n$ identity matrix, then $\lambda$ is an integrally normalized $U_p$-eigenvalue on  cohomology of $\Pi$.  We thank Chris Williams for explaining this to us.} we write 
\[
r_\lambda=\ord_p(\lambda)-\sum_{i=n+1}^{2n} h_i.
\]

\item[ii)]For $i\in \{n+2,\ldots,2n\}$, let $\Pr_i:\HIw(\Qp,T_\Pi)\rightarrow \HIw(\Qp,T_i/T_{i-1})$ denote the map induced by the projection $T_\Pi\rightarrow T_i/T_{i-1}$ as given in {Proposition}~\ref{thm:upper-tri}. For $i\in\{n,n+1\}$, we define $\Pr_i:\HIw(\Qp,T_\Pi)\rightarrow \HIw(\Qp,T_\Pi')$ similarly.
\item[iii)]For $\lambda\in\{\alpha,\beta\}$, we define the $\Lambda(\Gamma)$-morphism
\begin{align*}
    \cL^{(\lambda)}_{\underline\omega}:{\bigwedge}^n\HIw(\Qp,T_\Pi)&\rightarrow \cH_{r_\lambda}(\Gamma),\\
    z_1\wedge\cdots\wedge z_n&\mapsto\left(\prod_{i=n+1}^{2n} \tell_{h_i}\right)^{-1}\det(\cL_{\omega_i}\circ{\Pr}_i(z_j)),
\end{align*}
where the subscript $\underline\omega$ represents our choice of $\vp$-eigenvectors $\{\omega_i:i=n,n+1,\ldots,2n\}$ and the subscripts $i$ in the determinant run through the same subscripts in the product that defines $\lambda$. (The fact that the determinant is divisible by the product of $\tell_{h_i}$ follows from Proposition~\ref{lem:im1dim} and  Lemma~\ref{rk:divisible-log}{.)}

\item[iv)]We define similarly for $\bullet\in\{\#,\flat\}$ the Coleman maps
\begin{align*}
    \col^{\bullet}_{\underline\omega}:{\bigwedge}^n\HIw(\Qp,T_\Pi)&\rightarrow \Lambda(\Gamma),\\
    z_1\wedge\cdots\wedge z_n&\mapsto\det\left(\col_{\omega_{i}^\circ}\circ{\Pr}_i(z_j)|\col_{\omega_{n+2}}\circ{\Pr}_{n+2}(z_j)|\cdots |\col_{\omega_{2n}}\circ{\Pr}_{2n}(z_j)\right),
\end{align*}
where $i=n$ for $\bullet=\#$ and  $i=n+1$ for $\bullet=\flat$.
\end{defn}

\begin{remark}
It follows from \eqref{eq:decomp-PR-Col} that 
\[
\begin{pmatrix}
\cL^{(\alpha)}_{\underline\omega}\\\cL^{(\beta)}_{\underline\omega}
\end{pmatrix}=Q^{-1}M_{\log}'\begin{pmatrix}
\col^{\#}_{\underline\omega}\\\col^{\flat}_{\underline\omega}
\end{pmatrix}.
\]
\end{remark}

We finish this subsection by presenting the interpolation formulae of $\cL_{\underline\omega}^{(\lambda)}$.

For $i\in\{n+2,\ldots, 2n\}$, we write $\omega_i'$ for the dual basis of $\omega_i$ under the natural pairing
\[
\langle -,-\rangle:\Dcris(T_i/T_{i-1})\times\Dcris((T_i/T_{i-1})^*(1))\rightarrow\Dcris(\cO(1))=\cO\cdot t^{-1}e_1\rightarrow \cO.
\]
We write $\{\omega_n',\omega_{n+1}'\}$ for the dual basis of $\{\omega_n,\omega_{n+1}\}$ under the natural pairing
\[
    \langle -,-\rangle:\Dcris(T_n/T_{n-1})\otimes F\times\Dcris((T_n/T_{n-1})^*(1))\otimes F\rightarrow\Dcris(F(1))=F\cdot t^{-1}e_1\rightarrow F.
\]
Note that $\vp(\omega_i')=p^{-1}\alpha_i\omega_i'$.

\begin{proposition}\label{prop:interpolation}
Let $k$ be an integer such that $ h_{n+1}\le k\le h_n-1$  and $\theta$ a finite character on $\Gamma$ of conductor $p^m>1$. For $\lambda\in\{\alpha,\beta\}$ and $\bz=z_1\wedge\cdots \wedge z_n\in {\bigwedge}^n\HIw(\Qp,T_\Pi)$, we have
\[
\cL^{(\lambda)}_{\underline\omega}(\bz)(\cc^k\theta)=\left(\prod_{i=n+1}^{2n}\tell_{h_i}(\cc^k\theta)\right)^{-1}\det\left(\frac{k!p^{(m+1)k}}{\alpha_i^m\tau(\theta)}\langle\exp^*(e_\theta z_{i,j}\cdot e_{-k})\cdot t^{-k}e_k,\omega_i'\rangle\right),
\]
where $\tau(\theta)$ is the Gauss sum of $\theta$ and $e_\theta$ represents the element  $\displaystyle\sum_{\sigma\in \Gamma/\Gamma^{p^{m-1}}}\theta(\sigma)\sigma$ in the group ring $\Zp[\mu_{p^{m-1}}][\Gamma/\Gamma^{p^{m-1}}]$, the subscript $i$ in the determinant is indexed as in Definition~\ref{defn:PR}, the element $z_{i,j}$ denotes $\Pr_i(z_j)$ and  $\exp^*$ signifies the Bloch--Kato dual exponential map on $H^1(\Qp(\mu_{p^m}),T_i/T_{i-1})$ for $i\ge n+2$ (or $H^1(\Qp(\mu_{p^m}),T_\Pi')$ otherwise).
\end{proposition}
\begin{proof}
It follows from \cite[Theorem~B.5]{LZ0} that for $i\in\{n,n+2,\ldots 2n\}$ and $z\in\HIw(\Qp,T_i/T_{i-1})$, we have
\[
\cL_{T_{i}/T_{i-1}}(z)(\cc^k\theta)=\frac{j!p^{m(1+k)}}{\tau(\theta)}\vp^m\left(\exp^*(e_\theta z\cdot e_{-k})\cdot t^{-k}e_k\right).
\]
Therefore, 
\begin{align*}
\cL_{\omega_i}(z)(\cc^k\theta)&=\frac{j!p^{m(1+k)}}{\tau(\theta)}\langle\vp^m\left(\exp^*(e_\theta z\cdot e_{-k})\cdot t^{-k}e_k\right),\omega_i'\rangle\\
&=\frac{j!p^{m(1+k)}}{\tau(\theta)}\langle\exp^*(e_\theta z\cdot e_{-k})\cdot t^{-k}e_k,(p\vp)^{-m}(\omega_i')\rangle\\
&=\frac{j!p^{m(1+k)}}{\tau(\theta)}\langle\exp^*(e_\theta z\cdot e_{-k})\cdot t^{-k}e_k,\alpha_i^{-m}(\omega_i')\rangle    
\end{align*}
and the result follows.
\end{proof}

\subsection{Construction of bounded $p$-adic $L$-functions}\label{sec:padicL}
We introduce certain notions related to $p$-adic $L$-functions in order to prove Theorem~\ref{MT2}. We set
\[
\Crit(\Pi)=\{j\in\Z:\mu_{n+1}\le j\le \mu_n\}=\{j\in\Z:h_{n+1}-n+1\le j\le h_n-n\}.
 \] 
 By \cite[Section 1.1]{DJR}, the half integers $j +1/2$ for $j \in \Crit(\Pi)$ are precisely the critical points of the $L$-function $L(\Pi, s)$.
 In a {recent work of Barrera--Dimitrov--Williams \cite{BDW}},  $p$-adic $L$-functions $L_p^{(\lambda)}\in \cH_{r_\lambda}(\Gamma)$ are constructed for $\lambda\in\{\alpha,\beta\}$, where $\alpha$ and $\beta$ are defined as in Definition~\ref{defn:PR}. For an integer $j\in\Crit(\Pi)$ and a finite character $\theta$ on $\Gamma$ of conductor $p^m>1$, we have the interpolation formula:
 \[
 L_p^{(\lambda)}(\cc^j\theta)=\frac{c_{j,\theta}}{\lambda^m}L(\Pi\otimes\theta,j+1/2),
 \]
 where $c_{j,\theta}$ is a constant independent of the choice of $\lambda$ (but depends on $j$ and $\theta$).
 
\begin{remark}
For $\lambda\in\{\alpha,\beta\}$, recall $r_\lambda$ from Definition~\ref{defn:PR}. Our assumptions  (M.Slo) and (N.ord) combine to give $$\ord_p(\lambda)<\mu_n-\mu_{n+1}+1,$$ 
which is precisely the small slope bound condition, under which the construction of $p$-adic $L$-functions of Barrera--Dimitrov--Williams is valid unconditionally.\footnote{We thank Chris Williams for pointing this out to us.}
\end{remark}

 In order to relate these $p$-adic $L$-functions  to the Perrin-Riou maps and logarithmic matrix studied in the previous section, we introduce the following twisted $p$-adic $L$-functions.
 
 \begin{defn}\label{twistBWD}
 For $\lambda\in\{\alpha,\beta\}$ as in Definition~\ref{defn:PR}, the twisted $p$-adic $L$-functions are defined by
 \[
\sL_p^{(\lambda)}=\Tw^{n-1} L_p^{(\lambda)}.
 \]
 \end{defn}
 In particular, for $h_{n+1}\le k\le h_n-1$ and $\theta$ a finite character on $\Gamma$,
 \begin{equation}\label{eq:interpolation}
      \sL_p^{(\lambda)}(\cc^k\theta)= L_p^{(\lambda)}(\cc^j\theta)=\frac{c_{j,\theta}}{\lambda^m}L(\Pi\otimes\theta,j+1/2),
 \end{equation}
 where $j=k-n+1\in \Crit(\Pi)$.

We now prove Theorem~\ref{MT2}. 
 \begin{theorem}\label{thm:sharpflatpadicL}
There exist signed $p$-adic $L$-functions $\sL_p^\#,\sL_p^\flat\in\cH_0(\Gamma)=\Lambda(\Gamma)\otimes F$ such that
\begin{equation}\label{eq:factorization}
    \begin{pmatrix}
   \sL_p^{(\alpha)}\\ \sL_p^{(\beta)}
   \end{pmatrix}=Q^{-1}\Mlog'\begin{pmatrix}
 \sL_p^\#\\  \sL_p^\flat
   \end{pmatrix}.
\end{equation}
Furthermore,  at least one of the two signed $p$-adic $L$-functions, $\sL_p^\#$ and $\sL_p^\flat$, is non-zero.
\end{theorem}

\begin{proof}
By \cite[Proposition~2.11]{BFSuper}, the matrix $Q^{-1}\Mlog''$ satisfies the following property. For each $\mu\in\{\alpha_n,\alpha_{n+1}\}$, suppose that we are given $F_\mu\in \cH_{\ord_p(\lambda)-h_{n+1}}(\Gamma)$ such that for all $j\in\{0,\ldots,h_n-h_{n+1}-1\}$ and all Dirichlet characters $\theta$ of conductor $p^m>1$,
\[
F_\mu(\cc^j\theta)=\mu^{-m}\times c_{j,\theta}
\]
for some constant $c_{j,\theta}\in {\overline{\QQ}_p}$ that is independent of the choice of $\lambda$. Then there exist  $F_\#,F_\flat\in\cH_0(\Gamma)$ such that
\[
\begin{pmatrix}
F_{\alpha_n} \\ F_{\alpha_{n+1}}
\end{pmatrix}=Q^{-1}\Mlog''\cdot
\begin{pmatrix}
F_\#\\ F_\flat
\end{pmatrix}.
\]

For $\lambda\in\{\alpha,\beta\}$, let us write
 \[\fL_p^{(\lambda)}=\Tw^{-h_{n+1}}\sL_p^{(\lambda)}.\]
Then, we deduce from \eqref{eq:interpolation} that there exist $\fL_p^\#,\fL_p^\flat\in\cH_0(\Gamma)$ such that
\[
\begin{pmatrix}
   \fL_p^{(\alpha)}\\ \fL_p^{(\beta)}
   \end{pmatrix}=Q^{-1}\Mlog''\begin{pmatrix}
 \fL_p^\#\\  \fL_p^\flat
   \end{pmatrix}.
\]
Hence, we may take
 \[\sL_p^{\bullet}=\Tw^{h_{n+1}}\fL_p^{\bullet}\]
 for $\bullet\in\{\#,\flat\}$, proving \eqref{eq:factorization}.

 Our hypothesis (FL) implies that $|\Crit(\Pi)|>1$. As in \cite[proof of Proposition~{3.13}]{rob}, there exists $j\in \Crit(\Pi)$ such that  $L(\Pi\otimes\theta,j+1/2)\ne0$  by \cite[(1.3)]{JS}. In particular, both $\sL_p^{(\lambda)}$ are non-zero. Consequently, $\sL_p^\#$ and $\sL_p^\flat$ cannot be simultaneously zero by \eqref{eq:factorization}.
\end{proof}

After rescaling the periods in the construction of $\sL_p^{(\alpha)}$ and $\sL_p^{(\beta)}$ in \cite{BDW}, one may assert that $\sL_p^\#,\sL_p^\flat\in\Lambda(\Gamma)$. However, it is unclear to us whether there is an optimal choice of such periods.

\subsection{Reformulation of Rockwood's result}
In this section, we give an explicit description of the matrix $Q^{-1}\Mlog'$ under the Pollack condition, that is
\[
\alpha_n+\alpha_{n+1}=0.\tag{Pol}\label{eq:Pol}
\]
This allows us to recast Rockwood's plus and minus $p$-adic $L$-functions obtained in \cite{rob} in the framework of Theorem~\ref{thm:sharpflatpadicL}.

As we have already mentioned, (Pol) is a special case of the hypothesis (N.ord).
For an integer $m\ge1$, we define Pollack's half logarithms
\begin{align*}
\log_{p,m}^+&=\prod_{j=0}^{m-1}\frac{1}{p}\Tw^{-j}\left(\prod_{i=1}^\infty \frac{\Phi_{2i}(X)}{p}\right),  \\  
\log_{p,m}^-&=\prod_{j=0}^{m-1}\frac{1}{p}\Tw^{-j}\left(\prod_{i=1}^\infty \frac{\Phi_{2i-1}(X)}{p}\right).
\end{align*}
We recall from \cite[\S3]{rob} that there exist two bounded measures $L_p^\pm\in\cH_0(\Gamma)$ such that
\[
L_p^\pm=\frac{L_p^{(\alpha)}\pm L_p^{(\beta)}}{\Tw^{h_{n+1}-n+1}\log_{p,h_n-h_{n+1}}^\pm}.
\]
The construction of these $p$-adic $L$-functions follows closely the work of Pollack in \cite{pollack03}, where these functions were defined for normalized cuspidal eigen-newforms $f$on $\GL_2$ with $a_p(f)=0$. Let us define $\sL_p^\pm$ to be $\Tw^{n-1} L_p^\pm$. Then, by definition,
\[
    \sL_p^\pm=\frac{\sL_p^{(\alpha)}\pm \sL_p^{(\beta)}}{\Tw^{h_{n+1}}\log_{p,h_n-h_{n+1}}^\pm},
\]
which we may rewrite as an matrix equation:
\begin{equation}\label{eq:rob-pm}
   \begin{pmatrix}
   \sL_p^{(\alpha)}\\ \sL_p^{(\beta)}
   \end{pmatrix}=\frac{1}{2}\begin{pmatrix}
   \Tw^{h_{n+1}}\log_{p,h_n-h_{n+1}}^+&\Tw^{h_{n+1}}\log_{p,h_n-h_{n+1}}^-\\
    \Tw^{h_{n+1}}\log_{p,h_n-h_{n+1}}^+&-\Tw^{h_{n+1}}\log_{p,h_n-h_{n+1}}^-
   \end{pmatrix} \begin{pmatrix}
   \sL_p^+\\ \sL_p^-
   \end{pmatrix}.
\end{equation}

We shall recast \eqref{eq:rob-pm} in terms of the logarithmic matrix $Q^{-1}\Mlog'$. Under \eqref{eq:Pol}, the trace of the action of $\vp$ on $\Dcris(T_\Pi'')$ is zero. As in \cite[Proposition~5.10]{leiloefflerzerbes10}, the work of Berger--Li--Zhu \cite{bergerlizhu04} allows us to choose appropriate bases of $\Dcris(T_\Pi'')$ and $\NN(T_\Pi'')$ so that the matrix $\Mlog''$ is of the form
\begin{equation}
    \Mlog''=\begin{pmatrix}
    0&c^-\log_{p,h_n-h_{n+1}}^-\\
   c^+\log_{p,h_n-h_{n+1}}^+&0
    \end{pmatrix},
\end{equation}
where $c^\pm\in\cH_0(\Gamma_1)^\times$. Furthermore, we may choose $\beta_n=\beta_{n+1}=1$ in the matrix $Q$ (see Remark~\ref{rk:choice}), then 
\[
Q=\begin{pmatrix}
p^{h_{n+1}}\alpha_n^{-1}&-p^{h_{n+1}}\alpha_n^{-1}\\
1&1
\end{pmatrix}.
\]
We deduce that
\[
Q^{-1}\Mlog'=\frac{1}{2}\begin{pmatrix}
   d^+\Tw^{h_{n+1}}\log_{p,h_n-h_{n+1}}^+&d^-\Tw^{h_{n+1}}\log_{p,h_n-h_{n+1}}^-\\
    d^+\Tw^{h_{n+1}}\log_{p,h_n-h_{n+1}}^+&-d^-\Tw^{h_{n+1}}\log_{p,h_n-h_{n+1}}^-
   \end{pmatrix}
\]
for some $d^\pm\in\cH(\Gamma_1)^\times$. Therefore, we may rewrite \eqref{eq:rob-pm} as
\[
\begin{pmatrix}
   \sL_p^{(\alpha)}\\ \sL_p^{(\beta)}
   \end{pmatrix}=\frac{1}{2}Q^{-1}\Mlog'\begin{pmatrix}
   \frac{1}{d^+}\sL_p^+\\  \frac{1}{d^-}\sL_p^-
   \end{pmatrix}.
\]

{We conclude this section by the following remark. \begin{remark}Relying on the explicit description of Pollack's half logarithms, Rockwood proved in \cite[Proposition~3.13]{rob} that both $\sL_p^+$ and $\sL_p^-$ are non-zero, which is stronger  than the last assertion of Theorem~\ref{thm:sharpflatpadicL}. It is not clear to us how to prove this without \eqref{eq:Pol} since the description of the matrix $\Mlog'$ is much less explicit in the general case.\end{remark}}

\section{Iwasawa main conjectures and cohomological classes}\label{Selmer}
In this section, we define the so-called signed Selmer groups using the Coleman maps we defined in Section \ref{sec:col} and formulate Iwasawa main conjectures  relating them to the $p$-adic $L$-functions studied in Section \ref{sec:padicL}.
\subsection{Definitions of Selmer groups}
Let us first recall the definitions of the Bloch--Kato Selmer group  and the fine Selmer groups of Coates--Sujatha over a number field studied in \cite{blochkato} and \cite{CoatesSujatha_fineSelmer} respectively.
\begin{defn}\item[i)]Let $\Sigma$ be the set of prime numbers $\ell$ where $T_\Pi$ has bad reduction, as well as the prime $p$ and the archimedean prime in $\Q$.
Let $K$ be a number field and write $K_\Sigma$ for the maximal extension of $K$ that is unramified outside $\Sigma$.

\item[ii)]Write $T_\Pi^\dagger=\Hom_{\mathrm{cts}}(T_\Pi, F/\cO(1))$. The Bloch--Kato Selmer group of $T_\Pi^\dagger$ over $K$ is defined as $$\Sel_p(T_\Pi^\dagger/K):=\Ker\left(H^1(K_\Sigma/K,T_\Pi^\dagger) \rightarrow \prod_v\frac{H^1(K_v,T_\Pi^\dagger)}{H_f^1(K_v,T_\Pi^\dagger)}\right),$$
where $v$ runs through all the places of $K$ dividing $\Sigma$ (see \cite[Section 3]{blochkato} for the definition of the Bloch--Kato subgroups $H^1_f(K_v,T_\Pi^\dagger)$).

\item[iii)] The fine Selmer group of $T_\Pi^\dagger$ over $K$ is defined as $$\Sel_p^0(T_\Pi^\dagger/K):=\Ker\left(H^1(K_\Sigma/K,T_\Pi^\dagger) \rightarrow \prod_vH^1(K_v,T_\Pi^\dagger)\right),$$
where $v$ runs through all the places of $K$ dividing $\Sigma$.
\item[iv)]If $L$ is an infinite algebraic extension of $\Q$, we may define $\Sel_p(T_\Pi^\dagger/L)$ {and $\Sel_p^0(T_\Pi^\dagger/L)$} by the direct limits $\varinjlim\Sel_p(T_\Pi^\dagger/K)$ {and $\varinjlim\Sel_p^0(T_\Pi^\dagger/K)$ respectively}, where $K$ runs over finite sub-extension{s} of $L$.

\end{defn}
Notice that if $L$ is an infinite algebraic extension inside $\Q_\Sigma$, we have
$$\Sel_p(T_\Pi^\dagger/L)=\Ker\left(H^1(\Q_\Sigma/L,T_\Pi^\dagger) \rightarrow \prod_v\frac{H^1(L_w,T_\Pi^\dagger)}{H_f^1(L_w,T_\Pi^\dagger)}\right),$$
where $w$ runs through all places of $L$ dividing $\Sigma$ and $H_f^1(L_w,T_\Pi^\dagger)$ is given by $\varinjlim H_f^1(K_v,T_\Pi^\dagger)$, where $K$ runs through all finite sub-extension{s} of $L$ and $v$ denotes the place of $K$ lying below $w$. Similarly,
$$\Sel_p^0(T_\Pi^\dagger/L)=\Ker\left(H^1(\Q_\Sigma/L,T_\Pi^\dagger) \rightarrow \prod_v{H^1(L_w,T_\Pi^\dagger)}\right).$$

We now give an alternative description of the dual fine Selmer group $\Sel_p^0(T_\Pi^\dagger/K)^\vee$. By \cite[Section A.3]{perrinriou95}, we have the Poitou--Tate exact sequence 
\begin{equation*}\label{PTfinite}
\bigoplus_{v \in \Sigma}H^0(K_v,T_\Pi^\dagger) \rightarrow H^2(K_\Sigma/K,T_\Pi)^\vee \rightarrow H^1(K_\Sigma/K,T_\Pi^\dagger) \rightarrow \bigoplus_{v \in \Sigma}H^1(K_v,T_\Pi^\dagger).
\end{equation*}
Therefore, the fine Selmer group sits inside the following exact sequence:
\begin{equation*}
    \bigoplus_{v \in \Sigma}H^0(K_v,T_\Pi^\dagger) \rightarrow H^2(K_\Sigma/K,T_\Pi)^\vee \rightarrow \Sel_p^0(T_\Pi^\dagger/K) \rightarrow 0.
\end{equation*}
Taking Pontryagin duals and using  the fact that $$H^0(K_v,T_\Pi^\dagger)^\vee\cong H^2(K_v, T_\Pi),$$
we obtain 
\begin{equation}\label{FS_dual}
\Sel_p^0(T_\Pi^\dagger/K)^\vee = \Ker\left(H^2(K_\Sigma/K,T_\Pi)  \rightarrow \bigoplus_{v \in \Sigma}H^2(K_v, T_\Pi)\right).
\end{equation}

Set $\Q_\infty=\Q(\mu_{p^\infty})$ and $\Q_{p,k}=\Q_p(\mu_{p^k})$.
We now define  the signed Selmer groups of $T_\Pi^\dagger$ over  $\Q_\infty.$

\begin{defn}\label{def:sum_col}
\item[i)]For $\bullet\in\{\#,\flat\}$, define the direct sum counterpart of $\col^\bullet_{\underline\omega}$ to be
\begin{align*}
\tcol^\bullet_{\underline\omega}:\HIw(\Qp,T_\Pi)&\rightarrow\Lambda(\Gamma)^{\oplus n},\\
z&\mapsto    \col_{\omega_i^\circ}\circ{\Pr}_i\oplus \bigoplus_{j=n+2}^{2n}\col_{\omega_j}\circ {\Pr}_j,
\end{align*}
where $i=n$ for $\bullet=\#$ and $i=n+1$ for $\bullet=\flat$.
We write
\[
\HIw(\Qp, T_\Pi)^\bullet=\ker\left(\tcol^\bullet_{\underline\omega}\right).
\]
%and denote by $H^1(\Q_{p,k},T_\Pi)^\bullet$ the image of $\HIw(\Qp, T_\Pi)^\bullet$ under the projection $\HIw(\Qp, T_\Pi)\rightarrow H^1(\Q_{p,k},T_\Pi)$.
\item[ii)]Local Tate duality gives a perfect pairing
 \begin{equation}\label{localTate}
 \HIw(\Q_{p},T_\Pi) \times H^1(\Qp(\mu_{p^\infty}), T_\Pi^\dagger) \rightarrow F/\cO.
 \end{equation}
We define  $H^1_\bullet(\Q_{p}(\mu_{p^\infty}),T_\Pi^\dagger)$ to be the orthogonal complement  of  $\HIw(\Q_{p},T_\Pi)^\bullet$ under the Tate pairing \eqref{localTate}.
\end{defn}

 We can now finally define the ``signed" Selmer groups for $T_\Pi^\dagger$ as follows.
\begin{defn}\label{defn:signedSel}
Let $\bullet\in\{\#,\flat\}$. We define
$\Sel_p^\bullet(T_\Pi^\dagger/ \Q_\infty) $ to be the kernel of 
$$H^1(\QQ_\Sigma/\QQ_\infty,T_\Pi^\dagger)\rightarrow \frac{H^1(\Q_{p}(\mu_{p^\infty}),T_\Pi^\dagger)}{H^1_\bullet(\Q_{p}(\mu_{p^\infty}),T_\Pi^\dagger)}\times \prod_{v\nmid p}\frac{H^1(\Q_{\infty, v},T_\Pi^\dagger)}{H^1_f(\QQ_{\infty,v},T_\Pi^\dagger)}$$
where the last product runs through places of $\QQ_\infty$ dividing $\Sigma$ but not $p$.
\end{defn}

\begin{lemma}
  For $\bullet\in\{\#,\flat\}$ the $\Lambda(\Gamma)$-module $\Sel_p^\bullet(T_\Pi^\dagger/\Q_\infty)^{\vee}$ is finitely generated.
\end{lemma}
\begin{proof}
This follows from the fact that $H^1(\QQ_{\Sigma}/\QQ_\infty,T_\Pi^\dagger)^\vee$ is finitely generated over $\Lambda(\Gamma)$ (which is a result of Greenberg \cite[Proposition~3]{greenberg89}).
\end{proof}

We finish this subsection with an alternative description of $\HIw(\Qp,T_\Pi)^\bullet$.
Note that 
\[
\ker\left(\tcol^\bullet_{\underline\omega}\right)=\ker\left(\col_{\omega_i^\circ}\circ{\Pr}_i\right)\cap\bigcap_{j=n+2}^{2n}\ker\left(\col_{\omega_j}\circ{\Pr}_j\right)
\]
where $i$ is given as in Definition~\ref{def:sum_col}(i).
For  $n+2\le j \le 2n $,  recall from that $\ker\left(\col_{\omega_j}\right)=0$. This implies that $z\in \HIw(\Qp,T_\Pi)$ lies inside $\ker\left(\col_{\omega_j}\circ{\Pr}_j\right)$ if and only if its projection in $\HIw(\Qp,T_j/T_{j-1})$ is zero. Considering the filtration
\[
T_{n+1}\subset T_{n+2}\subset\cdots \subset T_{2n},
\]
we see that
\[
\bigcap_{j=n+2}^{2n}\ker\left(\col_{\omega_j}\circ{\Pr}_j\right)=\Image\left(\HIw(\Qp,T_{n+1})\rightarrow \HIw(\Qp,T_\Pi)\right).
\]

It remains to study $\ker\left(\col_{\omega_i^\circ}\circ{\Pr}_i\right)$. Let us write $\Pr'$ for the natural projection $\HIw(\Qp,T_\Pi)\rightarrow \HIw(\Qp,T_\Pi')$. Then, 
\[
\ker\left(\col_{\omega_i^\circ}\circ{\Pr}_i\right)=\left\{x\in\HIw(\Qp,T_\Pi):{\Pr}'(z)\in \ker\left(\col_{\omega_i^\circ}\right)\right\}.
\]
Therefore, we deduce that 
\begin{align*}
\HIw(\Qp,T_\Pi)^\#&=\big\{x\in\Image\left(\HIw(\Qp,T_n)\rightarrow\HIw(\Qp,T_\Pi)\right):{\Pr}'(z)\in \ker\col_{\omega_n^\circ}\big\},\\
\HIw(\Qp,T_\Pi)^\flat&=\big\{x\in\Image\left(\HIw(\Qp,T_n)\rightarrow\HIw(\Qp,T_\Pi)\right):{\Pr}'(z)\in \ker\col_{\omega_{n+1}^\circ}\big\}.
\end{align*}

When $\alpha_n=-\alpha_{n+1}$, (that is, when the Pollack condition \eqref{eq:Pol} holds), we may describe $\ker\col_{\omega_n^\circ}$ and $\ker\col_{\omega_{n+1}^\circ}$ explicitly in terms of  certain  jumping conditions on the classes in $H^1(\Q_{p,k},T_\Pi')$, similar to Kobayashi's  plus and minus  Selmer conditions for elliptic curves studied in  \cite{kobayashi03} as well as their generalizations to elliptic modular forms studied in \cite{lei11compositio}.

\subsection{Iwasawa Main Conjectures}

 We can now formulate the following ``signed" Iwasawa Main Conjecture. 
 \begin{conj}\label{conj:IMC} For $\bullet\in\{\#,\flat\}$ and $\eta$ a Dirichlet character modulo $p$, the $\Lambda(\Gamma_1)$-module $\Sel_p^\bullet(T_\Pi^\dagger/\Q_\infty)^{\vee, \eta}$ is torsion and 
 there exists an integer $k(\bullet,\eta)\geq 0$, depending on $\eta$ and $\bullet$, such that 
 $$\ch_{\Lambda(\Gamma_1)}(\Sel_p^\bullet(T_\Pi^\dagger/\Q_\infty)^{\vee, \eta})=(\pi^{k(\bullet,\eta)}\sL_p^{\bullet, \eta}/I_{\bullet,\eta}),$$
 where $\pi$ is a uniformizer of $F$ and $I_{\bullet,\eta}$ denotes a generator of $\ch_{\Lambda(\Gamma_1)}\left(\mathrm{coker}\left(\col^{\bullet}_{\underline\omega}\right)^\eta\right)$.
 \end{conj}
 
 \begin{remark}
 We expect that $\sL_p^{\bullet}$ to be inside the image of $\col^{\bullet}_{\underline\omega}$, which explains the presence of the term $I_{\bullet,\eta}$. See  Remark~\ref{rk:ES} in the next section for more details. {The appearance of $k(\bullet,\eta)$ is due to the possible non-integrality  of the $p$-adic $L$-functions. For any given $\bullet$ and $\eta$, one may set this constant to be zero after normalizing the   periods in the construction of Barrera--Dimitrov-Williams' $p$-adic $L$-functions. However, it is not clear to us whether this can be done simultaneously for all choices of $\bullet$ and $\eta$.} \end{remark}
 
We conclude this subsection by saying a few words on why we expect  $\Sel_p^\bullet(T_\Pi^\dagger/\Q_\infty)^{\vee, \eta}$ to be torsion over $\Lambda(\Gamma_1)$. We recall the following weak Leopoldt conjecture for $T_\Pi^\dagger$ (see \cite[Conjecture~3]{greenberg89}).

\begin{conj}\label{conj:WL}
The second cohomology group $H^2(\Q_\Sigma/\Q_\infty,T_\Pi^\dagger)$ is zero.
\end{conj}
By Remark \ref{conjugation} and \cite[Proposition~1.3.2]{perrinriou95}, Conjecture~\ref{conj:WL} is equivalent to 
\[
\rank_{\Lambda(\Gamma_1)}H^1(\Q_\Sigma/\Q_\infty,T_\Pi^\dagger)^{\vee,\eta}\stackrel{?}{=}n.
\]
As in \cite[proof of Proposition~6]{greenberg89}, 
\[
\rank_{\Lambda(\Gamma_1)}\left(\prod_{v\nmid p}\frac{H^1(\Q_{\infty, v},T_\Pi^\dagger)}{H^1_f(\QQ_{\infty,v},T_\Pi^\dagger)}\right)^{\vee,\eta}=0.
\]
By duality, we have
\[
\rank_{\Lambda(\Gamma_1)}\left(\frac{H^1(\Q_{p}(\mu_{p^\infty}),T_\Pi^\dagger)}{H^1_\bullet(\Q_{p}(\mu_{p^\infty}),T_\Pi^\dagger)}\right)^{\vee,\eta}=\rank_{\Lambda(\Gamma_1)}\left(\Ker\left(\tcol^\bullet_{\underline\omega}\right)\right)^\eta=n,
\]
since $\HIw(\Qp,T_\Pi)^\eta$ is of rank $2n$ over $\Lambda(\Gamma_1)$ (see \cite[Proposition in \S3.2.1]{perrinriou94}) and the image of each Coleman map in the direct sum defining $\tcol^\bullet_{\underline\omega}$ is non-zero (see Lemma~\ref{rk:injective} and Remark \ref{rk:PR2dim}(i)). Therefore, $\Sel_p^\bullet(T_\Pi^\dagger/\Q_\infty)^\eta$  is the kernel of a morphism from a $\Lambda(\Gamma_1)$-module of corank  conjecturally  $n$ to a $\Lambda(\Gamma_1)$-module of  corank $n$, which is a necessary (but not sufficient) condition to be cotorsion. This is  {in line} with the Greenberg Selmer group for $p$-ordinary representations studied in \cite{greenberg89}.

\subsection{Conjectural Euler systems}
We discuss in this section how the existence of an Euler system for the representation $T_\Pi$ would allow us to obtain evidence  towards Conjectures~\ref{conj:IMC}. We emphasize that the discussion in this section is mostly speculative. Throughout, we fix $\bullet\in\{\#,\flat\}$ and a Dirichlet character $\eta$ modulo $p$.

\begin{defn}
For $i=1,2$, define the $\Lambda(\Gamma)$-module
 $$\bH^i(T_\Pi) = \varprojlim_k H^i(\QQ_\Sigma/\Q(\mu_{p^k}), T_\Pi).$$
\end{defn}

Perrin-Riou formulated the following conjecture on the existence of Euler systems in \cite{pr-es}:

\begin{conj}\label{conj:PR}
Let $\cN$ be the set of integers of the form $mp^k$ where $m$ is a square-free product of integers that are coprime to the conductor of $T_\Pi$.  There exists a system of cohomological classes
\[
\left\{c_r\in\bigwedge^n_{\cO[G_r]}H^1(\QQ(\mu_r),T_\Pi):r\in\cN\right\}
\]
where $G_r=\Gal(\QQ(\mu_r)/\QQ)$ satisfying a precise norm relation as $m$ varies. Furthermore, the $p$-localization of $c_r$  is related to the complex $L$-values of $T_\Pi^*(1)$ twisted by characters on $G_r$ under the Bloch--Kato dual exponential map. {Furthermore, as $r$ varies, these classes are compatible under corestriction maps, up to multiplication by explicit Euler factors.}
\end{conj}
\begin{remark}
When the representation comes from {elliptic} modular forms, the existence of Euler systems is known; thanks to the work of Kato \cite{Kato}. There are also results on the existence of rank-one Euler systems (classes lying inside $H^1(\Q(\mu_r),T_\Pi)$, rather than in a wedge product) when
$G=\mathrm{GSp}(4)$ {in the ordinary case} (see \cite{LSZ,LZ-Gsp}) and  $G=\mathrm{GL}(2) \times \mathrm{GL}(2)$ in both ordinary and non-ordinary cases (see \cite{LeiLoefflerZerbes2014,KLZ1,KLZ2,LoefflerZerbes2016}). 
 \end{remark}
 
\begin{remark}\label{rk:ES}
Note that Conjecture~\ref{conj:PR}  predicts the existence of a special element $\bz=z_1\wedge\cdots \wedge z_n\in \bigwedge^n\bH^1(T_\Pi)$. It seems reasonable to expect the following equality to hold:
\begin{equation}
\cL_{\underline\omega}^{(\lambda)}(\loc(\bz))\stackrel{?}{=}\sL_p^{(\lambda)} ,   \label{eq:Rec}
\end{equation}
 $\lambda\in\{\alpha,\beta\}$. Here $\loc$ denotes the localization map 
\[
\bH^1(T_\Pi)\rightarrow \HIw(\Qp,T_\Pi),
\]
which we extend to the wedge products. {In fact, Proposition~\ref{prop:interpolation} gives us a hint on how the classes should be related to complex $L$-values under localizations and the Bloch--Kato dual exponential map.} For the rest of the article, we assume such an element $\bz$ does exist. 
Then \eqref{eq:Rec} would imply that
\[
\col_{\underline\omega}^\bullet(\loc(\bz))=\sL_p^\bullet
\]
{and it would  give an alternative and more direct proof of Theorem~\ref{MT2}.} 
Under various technical hypotheses of the Euler system machinery, it would also give the inclusion $\supset$ of Conjecture~\ref{conj:IMC}.
\end{remark}

The Poitou–Tate exact sequence in \cite[Proposition~A.3.2]{perrinriou95} gives the following exact sequence
\begin{align*}
H^1(\Q_\Sigma/\Q(\mu_{p^k}),T_\Pi) \rightarrow \frac{H^1(\Q_{p,k},T_\Pi)}{H^1(\Q_{p,k},T_\Pi)^\bullet}\oplus \bigoplus_{v \nmid p, v \in \Sigma} \frac{H^1(\Q(\mu_{p^k})_v,T_\Pi)}{H_f^1(\Q(\mu_{p^k})_v,T_\Pi)} \rightarrow \Sel_p^\bullet(T_\Pi^\dagger/\Q(\mu_{p^k}))^\vee& \\
\rightarrow H^2(\Q_\Sigma/\Q(\mu_{p^k}),T_\Pi) \rightarrow \bigoplus_{v \in \Sigma} H^2(\Q(\mu_{p^k})_v, T_\Pi)&.
\end{align*}
On combining this with \eqref{FS_dual}, we obtain  the exact sequence
\begin{align}
H^1(\Q_\Sigma/\Q(\mu_{p^k}),T_\Pi) \rightarrow \frac{H^1(\Q_{p,k},T_\Pi)}{H^1(\Q_{p,k},T_\Pi)^\bullet}\oplus \bigoplus_{v \nmid p, v \in \Sigma} \frac{H^1(\Q(\mu_{p^k})_v,T_\Pi)}{H_f^1(\Q(\mu_{p^k})_v,T_\Pi)} \rightarrow \Sel_p^\bullet(T_\Pi^\dagger/\Q(\mu_{p^k}))^\vee&\notag \\ 
\rightarrow \Sel_p^0(T_\Pi^\dagger/\Q(\mu_{p^k}))^\vee\rightarrow 0.& \label{eq:finite_level}
\end{align}
Upon taking inverse limits,  \cite[Section 17.10]{Kato} tells us that the modules in the second term of \eqref{eq:finite_level} vanish for $v \nmid p$.
Hence, after taking the $\eta$ component and inverse limits,\footnote{{Since the $\cO$-modules appearing in the Poitou--Tate exact sequence are finitely generated, the Mittag--Leffler condition is satisfied.}}  we obtain the following exact sequence
\begin{equation}\label{PT}
   \bH^1(T_\Pi)^\eta \rightarrow \frac{\HIw(\Q_p,T_\Pi)^\eta}{\HIw(\Q_{p},T_\Pi)^{\bullet,\eta}}\rightarrow \Sel_p^\bullet(T_\Pi^\dagger/\Q_\infty)^{\vee, \eta} \rightarrow  {\Sel_p^0(T_\Pi^\dagger/\Q_\infty)^{\vee, \eta} \rightarrow 0} ,
\end{equation}
where the second term is isomorphic to $\Image\left(\tcol_{\underline\omega}^\bullet\right)^\eta$ by the isomorphism theorem.  Let us write 
\[
\bZ(T_\Pi):=\Span_{\Lambda(\Gamma)}\{z_j\}_{j=1}^n\subset \bH^1(T_\Pi).
\]
Then \eqref{PT} gives the following exact sequence:
\begin{equation}\label{PT2}
0\rightarrow\frac{\bH^1(T_\Pi)^\eta}{\bZ(T_\Pi)^\eta} \rightarrow \frac{\Image\left(\tcol^{\bullet}_{\underline\omega}\right)^\eta}{\tcol^{\bullet}_{\underline\omega}\circ \loc(\bZ(T_\Pi))^\eta} \rightarrow \Sel_p^\bullet(T_\Pi^\dagger/\Q_\infty)^{\vee, \eta} \rightarrow {\Sel_p^0(T_\Pi^\dagger/\Q_\infty)^{\vee, \eta} \rightarrow 0}.
\end{equation}

In particular, we see that the equality of characteristic ideals in  Conjecture~\ref{conj:IMC} is, up to a power of  $\pi$, equivalent to 
\begin{equation}\label{eq:KatoIMC}
   \ch_{\Lambda(\Gamma_1)}\left(\frac{\bH^1(T_\Pi)^\eta}{\bZ(T_\Pi)^\eta}\right)\stackrel{?}{=}\ch_{\Lambda(\Gamma_1)}\left(\Sel_p^0(T_\Pi^\dagger/\Q_\infty)^{\vee, \eta}\right), 
\end{equation}
We note that $\Sel_p^0(T_\Pi^\dagger/\Q_\infty)^{\vee, \eta} \hookrightarrow  \bH^2(T_\Pi)^\eta$ and is hence torsion whenever  $\bH^2(T_\Pi)^\eta$ is torsion. As $T_\Pi$ is a finitely generated $\Z_p$-module, we have $$\varprojlim_k\bigoplus_{v \nmid p, v \in \Sigma}H^0(\kappa(v), H^1((\Q(\mu_{p^k})_v)_{nr}, T_\Pi))=0,$$
where $(\Q(\mu_{p^k})_v)_{nr}$ is the maximal unramified extension of  $\Q(\mu_{p^k})_v$ and $\kappa(v)$ is the residue field at $v$.
By an argument similar to \cite[p. 217]{kurihara02},  we have $$\ch_{\Lambda(\Gamma_1)}\left(\Sel_p^0(T_\Pi^\dagger/\Q_\infty)^{\vee, \eta}\right) = \ch_{\Lambda(\Gamma_1)}\left(\bH^2(T_\Pi)^\eta\right).$$ Hence \eqref{eq:KatoIMC} is equivalent to
\[
\ch_{\Lambda(\Gamma_1)}\left(\frac{\bH^1(T_\Pi)^\eta}{\bZ(T_\Pi)^\eta}\right)\stackrel{?}{=}\ch_{\Lambda(\Gamma_1)}\left(\bH^2(T_\Pi)^\eta\right).
\]
This is analogous to Kato's Iwasawa main conjecture  without $p$-adic zeta functions formulated for elliptic modular forms  (see \cite[Conjecture 12.10]{Kato}).

\begin{remark}
In this remark, we suppose that  Conjecture~\ref{conj:WL}  holds. By Remark \ref{conjugation} and \cite[Prop. 1.3.2]{perrinriou95},  $\bH^2(T_\Pi)^\eta$ would be $\Lambda(\Gamma_1)$-torsion, whereas the $\Lambda(\Gamma_1)$-rank of  $\bH^1(T_\Pi)^\eta$ would be $n$. Then \eqref{PT} tells us that  the torsionness of  $\Sel_p^\bullet(T_\Pi^\dagger/\Q_\infty)^{\vee, \eta} $  is in fact equivalent to the existence of an element $c_{p^\infty}\in \bigwedge^n\bH^1(T_\Pi)^n$ such that its image under $\col^{\bullet}_{\underline\omega}\circ \loc$ is non-zero.
\end{remark}

\section*{Data availability statement}
Data sharing not applicable to this article as no datasets were generated or analysed during the current study.

\bibliographystyle{amsalpha}
\bibliography{references, Iwasawa}
\end{document}